\documentclass[12pt, leqno]{amsart} 

\usepackage[mathcal]{euscript}
\usepackage{amssymb, amsfonts, amsthm, xy}
\usepackage{graphicx}
\usepackage{stackrel}

\usepackage{setspace}
\usepackage{tikz}
\usetikzlibrary{arrows,calc,decorations.markings,math,arrows.meta}
\usetikzlibrary{decorations.pathmorphing}
\usetikzlibrary{decorations.pathreplacing}
\usetikzlibrary{positioning}
\usetikzlibrary{shapes.geometric}
\usepgfmodule{oo}
\usepackage{pgflibraryarrows}
\usepackage{caption}
\usepackage{subcaption}
\usepackage{float}

\usepackage{pdfsync}

\xyoption{all} \CompileMatrices

\begin{document}

%%%%%%%%%%%%%%%%%%%%%%%%%%%%%%%%%%%%%   my definitions %%%%%%%%%%%%%%%%%%%

% tikz uses definitions below
\newcounter{n}
\setcounter{n}{0}
\newcounter{t}
\setcounter{t}{1}
\newcounter{d}
\setcounter{d}{2}
\newcounter{dt}
\setcounter{dt}{3}
\newcounter{nn}
\setcounter{nn}{-1}

\def\nodesize{3pt}

%%% end of tikz definitions

\def\ds{\displaystyle}
\def\arr{{\longrightarrow}}
\def\colim{\mathop{\rm colim}\nolimits}
\def\perm{\mathop{\rm perm}\nolimits}
\def\uperm{\mathop{\rm u\textnormal{-}perm}\nolimits}
\def\udet{\mathop{\rm u\textnormal{-}det}\nolimits}
\def\perfmatch{\mathop{\rm PerfMatch}\nolimits}
\def\pfaff{\mathop{\rm Pfaffian}\nolimits}
\def\sgn{\mathop{\rm sgn}\nolimits}

\def\intertitle#1{

\medskip

{\em \noindent #1}

\smallskip
}

\newtheorem{thm}{Theorem}[section]
\newtheorem{theorem}[thm]{Theorem}
\newtheorem{lemma}[thm]{Lemma}
\newtheorem{corollary}[thm]{Corollary}
\newtheorem{proposition}[thm]{Proposition}
\newtheorem{example}[thm]{Example}

\theoremstyle{definition}
\newtheorem{definition}[thm]{Definition}
\newtheorem{point}[thm]{}
\newtheorem{exercise}[thm]{Exercise}
\newtheorem{remark}[thm]{Remark}

\makeatletter
\let\c@equation\c@thm
\makeatother
\numberwithin{equation}{section}

%%%%%%%%%%%%%%%%%%%%%%%%%%%%%%%%%%%%%   temporary commands   %%%%%%%%%%%%%

%\doublespacing

\newcommand{\comment}[1]{}

%%%%%%%%%%%%%%%%%%%%%%%%%%%%%%%%%%%%%%%%%%%%%%%%%%%%  end of my stuff  %%%

\title{Undirected determinant, permanent and their complexity}
\author{Diana Dziewa-Dawidczyk, Adam J. Prze\'zdziecki}

\address{Diana Dziewa-Dawidczyk, Warsaw University of Life Sciences---SGGW, Warsaw, Poland}
\email{diana\_dziewa\_dawidczyk@sggw.edu.pl}
\address{Adam Prze\'zdziecki, Warsaw University of Life Sciences---SGGW, Warsaw, Poland}
\email{adamp@mimuw.edu.pl}

\maketitle
\begin{center}
\today
\end{center}

{\bf Abstract.} We view the determinant and permanent as functions on directed weighted graphs and introduce their analogues for the undirected graphs. We prove that the task of computing the undirected determinants as well as permanents for planar graphs, whose vertices have degree at most $4$, is \#P-complete. In the case of planar graphs whose vertices have degree at most $3$, the computation of the undirected determinant remains \#P-complete while the permanent can be reduced to the FKT algorithm, and therefore is polynomial.

The undirected permanent is a Holant problem and its complexity can be deduced from the existing literature. The concept of the undirected determinant is new. Its introduction is motivated by the formal resemblance to the directed determinant, a property that may inspire generalizations of some of the many algorithms which compute the latter.

For a sizable class of planar $3$-regular graphs, we are able to compute the undirected determinant in polynomial time.

\vspace{7pt}

\vspace{7pt}

\vspace{15pt}

\section{Introduction}
\label{section-introduction}

The most elegant definition of the {\em determinant} of an $n\times n$ matrix $A$ is probably the following one
\begin{equation}\label{equation-def-det}
  \det A = (-1)^n\sum_{c\in cc(A)}(-1)^{|c|}w(c)
\end{equation}
and the corresponding definition of the {\em permanent}
\begin{equation}\label{equation-def-perm}
  \perm A = \sum_{c\in cc(A)}w(c)
\end{equation}
The matrix $A$ above is viewed as an adjacency matrix of some weighted directed graph on $n$ vertices, denoted by the same letter $A$. The $cc(A)$ denotes the set of cycle covers of $A$. A {\em cycle cover} $c\in cc(A)$ is a subgraph which contains all the vertices of $A$ and every vertex of $c$ has both in-degree and out-degree equal 1. The symbol $|c|$ denotes the number of connected components of $c$. The symbol $w(c)$ is the weight of $c$, that is, the product of the weights of all the edges in $c$.

We see that the definitions above make sense when $A$ is an undirected graph. Thus we define the {\em undirected determinant} and the {\em undirected permanent} using the formulas \ref{equation-def-det} and \ref{equation-def-perm}, respectively, applied to an undirected graph $A$. In the case of the undirected graph, the cycle covers are also called {\em $2$-factors}, thus the undirected permanent counts the weighted $2$-factors of a graph and the undirected determinant counts them with the sign depending on the parity of the number of components.

{\bf We prove in this paper} that the task of computing the undirected determinants as well as permanents for planar graphs, whose vertices have degree at most $4$, is \#P-complete, see Theorem \ref{theorem-degree-4}. In the case of planar graphs whose vertices have degree at most $3$, the computation of the undirected permanent can be reduced, see Section \ref{section-permanent-3}, to the Fisher-Kasteleyn-Temperley (FKT) algorithm \cite{temperley-fisher,kasteleyn-a,kasteleyn-b}, and therefore is computable in polynomial time.
This result contrasts with the next one, Theorem \ref{theorem-degree-3}, which establishes the \#P-completeness of the undirected determinant of the $3$-regular planar graphs. However, for a sizable class of planar $3$-regular graphs (actually, we barely peek beyond the bipartite graphs) we are able to compute the undirected determinant in polynomial time, see Theorem \ref{theorem-undirected-det-computable}.

The undirected permanent is an instance of a symmetric Holant problem and its complexity can be deduced from the existing literature, see Cai, Fu, Guo and Williams \cite{holant-dichotomy}. Nevertheless, we prove the results for the permanent alongside with those for the determinant, as while adding little extra work, it allows to emphasize the strain between the polynomial and the \#P-complete.

The concept of the undirected determinant is new.
The alterations of the definition of the determinant have a long history.
The non-commutative generalization of determinant, called the Cayley determinant, is \#P-complete even over the ring of $2\times 2$ matrices, see \cite{CHSS11}. The Fermionants and the immanants also tend to be \#P-complete, see \cite{salvation,fermionant-immanant}. It would seem that almost any modification of the definition of determinant leads to a \#P-complete polynomial.

We were guided by the intuition that the main reason that causes these alterations to be \#P-complete is their deprivation of some of the very special properties enjoyed by the determinant.
More specifically, we suspect that the key limitations of the FKT algorithm stem from its direct dependence on the determinant. Linear properties of the determinant cause the FKT algorithm to count only those structures which allow some kind of ``interpolation''. To illustrate what we have in mind, let us recall that Valiant introduced matchgates \cite{valiant-matchgates,valiant-matchgates-quantum}. Then, the characterization of the possible signatures of planar matchgates has been completed by Cai and Gorenstein \cite{matchgates-revisited} in terms of Matchgate Identities. These identities imply that for any matchgate $G$, if its signature $\Gamma_G$ is non-zero on two length-$k$ bitstrings $\alpha,\beta\in\{0,1\}^k$ then there exists a sequence
$$
  \alpha=\alpha_0,\alpha_1,\ldots,\alpha_s=\beta
$$
in $\{0,1\}^k$ such that $\alpha_{i-1}$ differs from $\alpha_i$, at exactly two places, for $i=1,2,\ldots,s$ and $\Gamma_G^{\alpha_i}\neq 0$ for $i=0,1,\ldots,s$. This contrasts with the behavior of Boolean formulas: the information that two assignments $(x_1,x_2,\ldots,x_n)$ and $(y_1,y_2,\ldots,y_n)$ satisfy a formula $\phi$ does not imply that any other assignment will satisfy $\phi$.

The reach of the FKT algorithm as been extended by the holographic reductions, introduced by Valiant \cite{valiant-accidental,valiant-holographic}. These, at least in large part, are a way to incorporate the linear properties of the determinant into the formulation of the combinatorial problems.

The authors have attempted to avoid at least some limitations, related to the linear properties of the determinant, by altering one of the many algorithms for determinant rather than its definition. The most promising candidate seemed to be, probably the most combinatorial one, known as the Mahajan-Vinay (MV) algorithm \cite{mahajan-vinay}. Rote \cite{rote} discussed MV and a number of related, but less combinatorial in nature, algorithms.

While the MV algorithm received a lot of attention in the contexts of algebraic branching programs or circuits, its other applications seem to be unexplored. For example an observation that the MV algorithm computes the non-commutative Moore determinant of the self-adjoint quaternionic matrices, has not yet been published. See \cite{dyson} for a discussion of the Moore determinant. The authors were inspired by an observation that a slight modification of the MV algorithm computes the undirected permanent of planar graphs whose vertices have degree at most $3$. In Section \ref{section-permanent-3}, we compute the same permanents using the FKT algorithm. This led us to the investigation of the undirected determinant, which, as we mentioned above, turned out to be \#P-complete. While the alteration of signs in the directed determinant enables cancellations that make it possible to construct polynomial algorithms for determinant, the alteration of signs in the undirected determinant might suggest some amendments to the algorithms for determinant. To some extent we follow in this direction in Section \ref{section-semi-pfaffian}.

The MV algorithm is mentioned only in the introduction, above, as the source of our inspiration. It is not present in the final writing of this paper -- we simplified the arguments by expressing them in the language of the FKT algorithm. However, the dependence of our algorithms on the linear properties of determinant is slightly weakened by applying the Pfaffian to a matrix of inverses, see Section \ref{section-permanent-3} and Theorem \ref{theorem-undirected-det-computable}.

\section{Notation and preliminaries}

Most of our terminology is standard and follows for example Thomas \cite{thomas}. All graphs are finite, undirected, althought, in Section \ref{section-semi-pfaffian} we consider undirected graphs whose edges are equipped with an orientation. The graphs may have loops or multiple edges, however, each case when they may occur is mentioned explicitly in the text. In each such case either we explicitly show how to modify our graphs to remove loops and multiple edges, or these are used only as parts of a proof.

The symbols $V(G)$ and $E(G)$ denote the sets of vertices and edges of the graph $G$. We write $v\in G$ or $e\in G$ instead of $v\in V(G)$ or $e\in E(G)$ when no confusion can arise. Elements of $E(G)$ are denoted $\{v_1,v_2\}$, or $v_1v_2$ if we want to indicate that the edge is oriented from $v_1$ to $v_2$. The graph $G$ is {\em weighted}, which means that it is equipped with the weight function $w:E(G)\arr F$. The reader may view the $F$ as the rational numbers, we are going to use only the weights $1$, $-1$ and $-{1\over 2}$, however, in the constructions, $F$ may be any field of characteristics not equal $2$.

A {\em cycle cover} of $G$, denoted $cc(G)$, is a spanning subgraph whose all vertices have degree $2$. In the literature, the cycle covers are also called vertex cycle covers or $2$-factors. The weight of a cycle cover $c\in cc(G)$ is the product of the weights of its edges: $w(c)=\prod_{e\in c}w(e)$.

The {\em undirected determinant} and the {\em undirected permanent} are defined as
\begin{equation}\label{equation-def-det}
  \udet G = (-1)^n\sum_{c\in cc(G)}(-1)^{|c|}w(c)
\end{equation}
and
\begin{equation}\label{equation-def-perm}
  \uperm G = \sum_{c\in cc(G)}w(c)
\end{equation}
where $n$ is the number of vertices in $G$. These definitions make sense also in the case of graphs with loops and multiple edges.

A {\em gadget} is a graph $G$ equipped with a set of {\em external edges} $ext(G)$ such that each $e\in ext(G)$ is adjacent to exactly one vertex in $G$. When referring to gadgets, we use the term {\em cycle} to mean either the actual cycle or a path connecting two external edges. Each cycle cover $c\in cc(G)$ of a gadget $G$ determines a, possibly empty, subset of $ext(G)$ consisting of those edges in $ext(G)$ which belong to $c$.

A {\em signature} of a gadget $G$ is a function
$$\mathop{\rm signature}:\mathcal{P}(ext(G))\to F$$
where $\mathcal{P}(ext(G))$ denotes the set of all subsets of $ext(G)$. The {\em permanental signature} of a gadget $G$ is defined as
$$
  \mathop{\rm signature}(S)=\sum_{\substack{c\in cc(G) \\ c\mskip 1mu \cap ext(G)=S}}w(c)
$$
while the {\em determinantal signature} is
$$
  \mathop{\rm signature}(S)=(-1)^n\sum_{\substack{c\in cc(G) \\ c\mskip 1mu \cap ext(G)=S}}(-1)^{|c|+\tau}w(c)
$$
where $n$ is the number of vertices of $G$ and $|c|$ is the number of components of $c$. The number $\tau$ is $0$ if $|S|<4$. For larger subsets $S$ the number $\tau$ is defined relative to a fixed pairing $P_0$ of the elements of $S$. For each cycle cover $c\in cc(G)$ those components that are not cycles, connect a pair of elements of $S$ which yields another pairing $P_c$. If $P_c=P_0$ we define $\tau=0$. Otherwise we can connect $P_0$ to $P_c$ by a sequence of modifications, where each of them involves two pairs. These are of the form: $\{\{a,b\},\{c,d\}\}\mapsto\{\{a,d\},\{c,b\}\}$.
For every such modification we change $\tau$ to $1-\tau$.
In the simplest, and actually the only one which is important here, case when $S=ext(G)=\{a,b,c,d\}$ this can be illustrated as in {\bf Figure \ref{figure-pairing}} below.
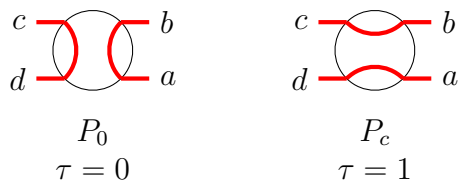
\begin{figure}[H]
\begin{tikzpicture}[scale=0.75]
  \begin{scope}[shift={(0,0)}]
    \draw (0,1) arc (135:495:0.7071);
    \draw[red, ultra thick] (0,1) arc (45:-45:0.7071);
    \draw[red, ultra thick] (1,1) arc (135:225:0.7071);

    \draw[red, ultra thick] (0,1)--++(-0.5,0);
    \draw[red, ultra thick] (0,0)--++(-0.5,0);
    \draw[red, ultra thick] (1,1)--++(0.5,0);
    \draw[red, ultra thick] (1,0)--++(0.5,0);

    \draw (1.5,0) node [anchor=west] {$a$};
    \draw (1.5,1) node [anchor=west] {$b$};
    \draw (-0.5,1) node [anchor=east] {$c$};
    \draw (-0.5,0) node [anchor=east] {$d$};

    \node at (0.5,-0.95) {$P_0$};
    \node at (0.5,-1.6) {$\tau=0$};
  \end{scope}

  \begin{scope}[shift={(5,0)}]
    \draw (0,1) arc (135:495:0.7071);
    \draw[red, ultra thick] (1,1) arc (-45:-135:0.7071);
    \draw[red, ultra thick] (0,0) arc (135:45:0.7071);

    \draw[red, ultra thick] (0,1)--++(-0.5,0);
    \draw[red, ultra thick] (0,0)--++(-0.5,0);
    \draw[red, ultra thick] (1,1)--++(0.5,0);
    \draw[red, ultra thick] (1,0)--++(0.5,0);

    \draw (1.5,0) node [anchor=west] {$a$};
    \draw (1.5,1) node [anchor=west] {$b$};
    \draw (-0.5,1) node [anchor=east] {$c$};
    \draw (-0.5,0) node [anchor=east] {$d$};

    \node at (0.5,-0.95) {$P_c$};
    \node at (0.5,-1.6) {$\tau=1$};
  \end{scope}

\end{tikzpicture}
\caption{Modification of a pairing of external edges inside of a gadget.}
\label{figure-pairing}
\end{figure}

Suppose that a gadget $G$ is a part of an abient graph and $c$ is a cycle cover of the graph. Each modification of $P_{c\cap G}$ changes the parity of the number of those cycles in $c$ which are not entirely inside of $G$. This way we see that $\tau$ is well defined. Also, when computing the undirected determinant of the ambient graph, we may ignore the inner structure of $G$ and use only its signature and assume that the cycle traverses $G$ in the way indicated by $P_0$.

By abuse of terminology, we often write ``signature'' when we mean the value of the signature at a specific subset. When we list the values of the signature we omit subsets at which the signature is zero.

\section{The undirected determinant and permanent of planar graphs of maximum degree $4$ are \#P-complete}
\label{section-degree-4}

In this section, for a given Boolean formula $\phi(x_1,x_2,\ldots,x_n)$, in conjunctive normal form with $m$ clauses, where each clause is limited to three literals, we construct an undirected weighted planar graphs $A_\phi$ and $B_\phi$, of maximum degree $4$, such that:
\begin{itemize}
  \item the number of assignments satisfying $\phi$,
  \item $(-1)^m$ times the undirected determinant of $A_\phi$,
  \item the undirected permanent of $B_\phi$
\end{itemize}
are all equal. The graphs $A_\phi$ and $B_\phi$ differ only by signs of weights of some edges -- thus it is convenient to describe them in parallel. The size of these graphs is bounded by $O(m^2)$.

We construct these graphs by means of the following gadgets:
\begin{enumerate}
  \item The skew crossover gadget.
  \item The iff gadget -- it synchronizes two edges which belong to the boundary of a common face of a planar graph. Connecting the two edges with an iff gadget results in a new graph whose permanent (respectively determinant) counts the weights of only those cycle covers, of the original graph, which contain either both or none of the edges connected by the gadget.
  \item The extended iff gadget -- synchronizes any two edges of a planar graph.
  \item The variable setting gadget -- encodes the variables of the formula $\phi$.
  \item The clause gadget -- encodes the clauses of the formula $\phi$.
\end{enumerate}

{\bf (1) The skew crossover gadget.} We start with a skew crossover gadget, shown in {\bf Figure \ref{figure-skew-crossover}}. It is inspired by the Cai-Gorenstein \cite{matchgates-revisited} construction. Its signature is $0$ unless the opposite edges, either both or none, belong to a cycle cover. The nonzero signatures are either $1$ or $-1$.
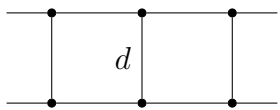
\begin{figure}[H]
\def\a{0.2}
\begin{tikzpicture}[scale=1.2]
  \tikzstyle{vertex}=[circle,minimum size=\nodesize,inner sep=0pt,draw=black,fill]

  \foreach \name/\x/\y in {a/-1/2,b/0/2,c/1/2,d/-1/1,e/0/1,f/1/1}
    \node[vertex] (\name) at (\x,\y) {};
  \foreach \from/\to in {a/b,b/c,d/e,e/f,a/d,b/e,c/f}
    \draw (\from) -- (\to);
  \draw (a) -- ++(-0.5,0);
  \draw (c) -- ++(0.5,0);
  \draw (d) -- ++(-0.5,0);
  \draw (f) -- ++(0.5,0);
  \draw (0,1.5) node [anchor=east] {$d$};
\end{tikzpicture}
\caption{The skew crossover gadget}
\label{figure-skew-crossover}
\end{figure}
\noindent The weight $d$ is defined as $d=1$ for determinant and $d=-1$ for permanent.

In {\bf Figure \ref{figure-skew-crossover-signature}}, we list all the possible ways a cycle cover may meet this gadget, and the ways they add to yield the signature of this gadget, in the permanental case.

\def\s{1}
\def\ext{0.2}
% box symbol of the gadget. (x,y)pos = lower left corner
\newcommand\xgadget[9]{ % xpos,ypos,1,2,3,4- thick/not,
% 1/2 - arc/red or 3/4 straight/red , bottom label, right label
  \draw (#1+1,#2+0.5) node {x};
  \draw (#1-\ext,#2-\ext) -- ++(2*\s+2*\ext,0) -- ++(0,\s+2*\ext) -- ++(-2*\s-2*\ext,0) -- ++(0,-\s-2*\ext);
  \draw\ifnum\value{#3}=1[red, ultra thick]\fi
    (#1-\ext,#2+1) -- ++(-0.5,0);
  \draw\ifnum\value{#4}=1[red, ultra thick]\fi
    (#1+2+\ext,#2+1) -- ++(0.5,0);
  \draw\ifnum\value{#5}=1[red, ultra thick]\fi
    (#1-\ext,#2) -- ++(-0.5,0);
  \draw\ifnum\value{#6}=1[red, ultra thick]\fi
    (#1+2+\ext,#2) -- ++(0.5,0);
  \ifnum#7=1
    \draw[-] (#1-\ext,#2+1) arc (90:-90:0.5);
    \draw[-] (#1+2+\ext,#2+1) arc (90:270:0.5);
  \fi
  \ifnum#7=2
    \draw[-,red, ultra thick] (#1-\ext,#2+1) arc (90:-90:0.5);
    \draw[-,red, ultra thick] (#1+2+\ext,#2+1) arc (90:270:0.5);
  \fi
  \ifnum#7=3
    \draw[-] (#1-\ext,#2+1) -- ++(2.5,0);
    \draw[-] (#1-\ext,#2) -- ++(2.5,0);
  \fi
  \ifnum#7=4
    \draw[-,red, ultra thick] (#1-\ext,#2+1) -- ++(2.5,0);
    \draw[-,red, ultra thick] (#1-\ext,#2) -- ++(2.5,0);
  \fi

  \draw (#1+1,#2-0.2) node [anchor=north] {#8};
  \draw (#1+3,#2+0.5) node {#9};
}

% top and center part of the detailed gadget
\newcommand\bgadget[9]{ % xpos,ypos (bottom left corner)
  \tikzstyle{vertex}=[circle,minimum size=\nodesize,inner sep=0pt,draw=black,fill]

  \foreach \name/\x/\y in {a/0/1,b/1/1,c/2/1,d/0/0,e/1/0,f/2/0}
    \node[vertex] (\name) at (#1+\x,#2+\y) {};

  \draw\ifnum\value{#3}=1[red, ultra thick]\fi (a) -- ++(-0.5,0);
  \draw\ifnum\value{#4}=1[red, ultra thick]\fi (a) -- (b);
  \draw\ifnum\value{#5}=1[red, ultra thick]\fi (b) -- (c);
  \draw\ifnum\value{#6}=1[red, ultra thick]\fi (c) -- ++(0.5,0);
  \draw\ifnum\value{#7}=1[red, ultra thick]\fi (a) -- (d);
  \draw\ifnum\value{#8}=1[red, ultra thick]\fi (b) -- (e);
  \draw\ifnum\value{#9}=1[red, ultra thick]\fi (c) -- (f);
}

% bottom part of the detailed gadget
\newcommand\cgadget[8]{% xpos,ypos, 3,4,5,6: t/n, bottom label, right label
  \tikzstyle{vertex}=[circle,minimum size=\nodesize,inner sep=0pt,draw=black,fill]

  \foreach \name/\x/\y in {a/0/1,b/1/1,c/2/1,d/0/0,e/1/0,f/2/0}
    \node[vertex] (\name) at (#1+\x,#2+\y) {};

  \draw\ifnum\value{#3}=1[red, ultra thick]\fi (d) -- ++(-0.5,0);
  \draw\ifnum\value{#4}=1[red, ultra thick]\fi (d) -- (e);
  \draw\ifnum\value{#5}=1[red, ultra thick]\fi (e) -- (f);
  \draw\ifnum\value{#6}=1[red, ultra thick]\fi (f) -- ++(0.5,0);

  \draw (#1+1,#2-0.1) node [anchor=north] {#7};
  \draw (#1+3,#2+0.5) node {#8};
}

%%%%%%%%%%%%%%%%%%%%%%%%%%%%%%%%%%%%%%%%%%%%%%%%%%%%%%%%%%
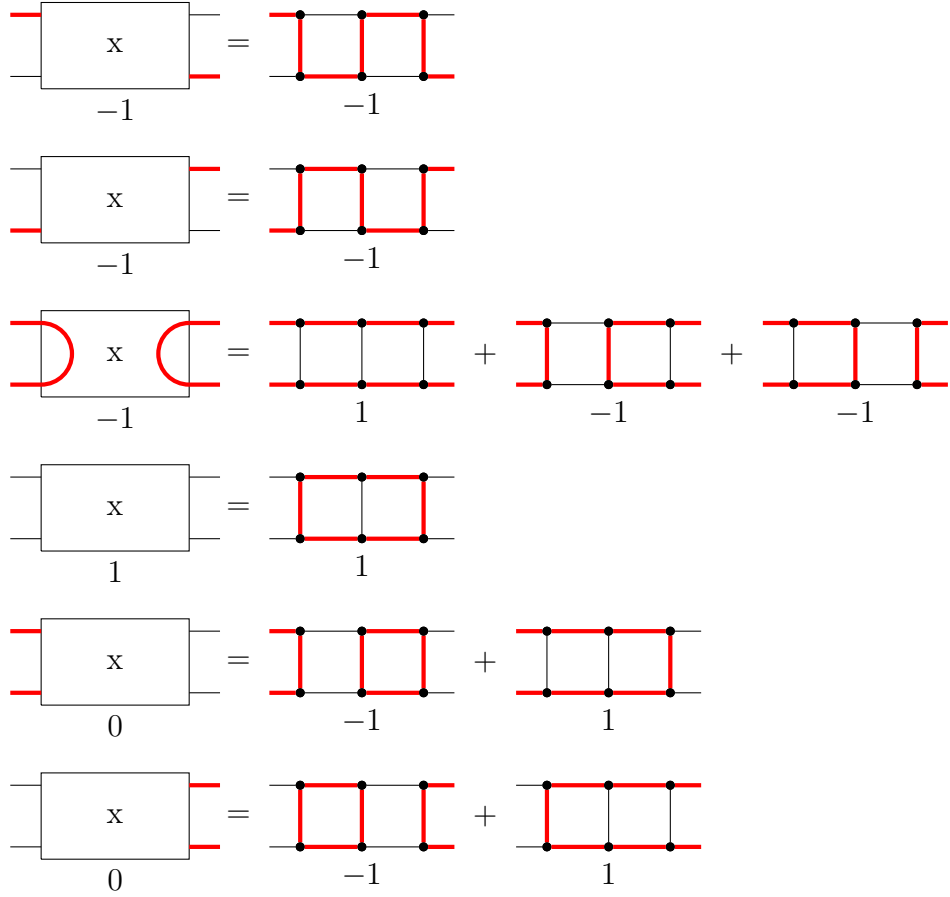
\begin{figure}[H]
\begin{tikzpicture}[scale=0.82]
  \def\yy{0}
  \xgadget{0}{\yy}{t}{n}{n}{t}{0}{$-1$}{$=$}
  \bgadget{4}{\yy}{t}{n}{t}{n}{t}{t}{t}
  \cgadget{4}{\yy}{n}{t}{n}{t}{$-1$}{}

  \def\yy{-2.5}
  \xgadget{0}{\yy}{n}{t}{t}{n}{0}{$-1$}{$=$}
  \bgadget{4}{\yy}{n}{t}{n}{t}{t}{t}{t}
  \cgadget{4}{\yy}{t}{n}{t}{n}{$-1$}{}

  \def\yy{-5}
  \xgadget{0}{\yy}{t}{t}{t}{t}{2}{$-1$}{$=$}
  \bgadget{4}{\yy}{t}{t}{t}{t}{n}{n}{n}
  \cgadget{4}{\yy}{t}{t}{t}{t}{$1$}{$+$}
  \bgadget{8}{\yy}{t}{n}{t}{t}{t}{t}{n}
  \cgadget{8}{\yy}{t}{n}{t}{t}{$-1$}{$+$}
  \bgadget{12}{\yy}{t}{t}{n}{t}{n}{t}{t}
  \cgadget{12}{\yy}{t}{t}{n}{t}{$-1$}{}

  \def\yy{-7.5}
  \xgadget{0}{\yy}{n}{n}{n}{n}{0}{$1$}{$=$}
  \bgadget{4}{\yy}{n}{t}{t}{n}{t}{n}{t}
  \cgadget{4}{\yy}{n}{t}{t}{n}{$1$}{}

  \def\yy{-10}
  \xgadget{0}{\yy}{t}{n}{t}{n}{0}{$0$}{$=$}
  \bgadget{4}{\yy}{t}{n}{t}{n}{t}{t}{t}
  \cgadget{4}{\yy}{t}{n}{t}{n}{$-1$}{$+$}
  \bgadget{8}{\yy}{t}{t}{t}{n}{n}{n}{t}
  \cgadget{8}{\yy}{t}{t}{t}{n}{$1$}{}

  \def\yy{-12.5}
  \xgadget{0}{\yy}{n}{t}{n}{t}{0}{$0$}{$=$}
  \bgadget{4}{\yy}{n}{t}{n}{t}{t}{t}{t}
  \cgadget{4}{\yy}{n}{t}{n}{t}{$-1$}{$+$}
  \bgadget{8}{\yy}{n}{t}{t}{t}{t}{n}{n}
  \cgadget{8}{\yy}{n}{t}{t}{t}{$1$}{}
\end{tikzpicture}
\caption{The signature of the skew crossover gadget in the case of permanent, the determinantal signature has the opposite sign}
\label{figure-skew-crossover-signature}
\end{figure}
%%%%%%%%%%%%%%%%%%%%%%%%%%%%%%%%%%%%%%%%%%%%%%%%%%%%%%%%%%

In the determinantal case, all values of the signature of this gadget have the opposite sign, however, the third one requires an additional comment. This is the only instance, where the parity of the number of components of the cycle cover depends on the way it meets the gadget.
The arcs inside the box indicate the way the cycles should traverse this gadget, for the sake of counting the parity of the number of cycles in a cover.

{\bf (2) The iff gadget}, shown in {\bf Figure \ref{figure-iff}}, has $4$ external edges and its signature is $1$ if either all or none of them belong to a cycle cover. Otherwise its signature is $0$. All other cycle covers cancel out. The gadget is used to synchronize two edges which belong to the boundary of a common face of the planar graph.

The iff gadget can be constructed, using the skew crossover gadget, as shown in {\bf Figure \ref{figure-iff}}, below.

\def\s{1}
\def\ext{0.2}
% box symbol of the gadget. (x,y)pos = lower left corner
\newcommand\iffgadget[5]{ % xpos,ypos, 3 - thick/not, t/n, bottom label, right label
  \tikzstyle{vertex}=[circle,minimum size=\nodesize,inner sep=0pt,draw=black,fill]

  \node[vertex] at (#1,#2+0.5*\s) {};
  \node[vertex] at (#1+2*\s,#2+0.5*\s) {};

  \draw (#1+\s,#2+0.5*\s) node {iff};
  \draw (#1+0.5*\s,#2) -- ++(-0.5*\s,0.5*\s) -- ++(0.5*\s,0.5*\s)
    -- ++(\s,0) -- ++(0.5*\s,-0.5*\s) -- ++(-0.5*\s,-0.5*\s) -- ++(-\s,0);
  \draw\ifnum\value{#3}=1[red, ultra thick]\fi
    (#1,#2+0.5*\s) -- ++(-0.25*\s,0.25*\s);
  \draw\ifnum\value{#3}=1[red, ultra thick]\fi
    (#1,#2+0.5*\s) -- ++(-0.25*\s,-0.25*\s);
  \draw\ifnum\value{#3}=1[red, ultra thick]\fi
    (#1+2*\s,#2+0.5*\s) -- ++(0.25*\s,0.25*\s);
  \draw\ifnum\value{#3}=1[red, ultra thick]\fi
    (#1+2*\s,#2+0.5*\s) -- ++(0.25*\s,-0.25*\s);

  \draw (#1+1,#2-0.2) node [anchor=north] {#4};
  \draw (#1+3,#2+0.5) node {#5};
}

\def\xx{5.5}

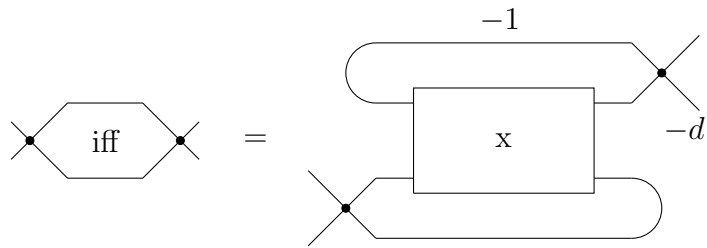
\begin{figure}[H]

\begin{tikzpicture}[scale=1]
  \iffgadget{0}{0}{n}{}{$=$}
  \xgadget{\xx-0.2}{0}{n}{n}{n}{n}{0}{}{}

  \draw (\xx+2.5,0) arc (90:-90:2*\ext);
  \draw (\xx-0.5-2*\ext,-4*\ext) -- (\xx+2.5,-4*\ext);

  \draw (\xx-0.5-2*\ext,-4*\ext) -- ++(-2*\ext-0.5,2*\ext+0.5);
  \draw (\xx-0.5-2*\ext,0) -- ++(-2*\ext-0.5,-2*\ext-0.5);
  \node[vertex] at (\xx-0.9-2*\ext,-2*\ext) {};

  \draw (\xx-0.5-2*\ext,1) arc (270:90:2*\ext);
  \draw (\xx-0.5-2*\ext,1+4*\ext) -- (\xx+2.5,1+4*\ext);

  \draw (\xx+2.5,1) -- ++(2*\ext+0.5,2*\ext+0.5);
  \draw (\xx+2.5,1+4*\ext) -- ++(2*\ext+0.5,-2*\ext-0.5);
  \node[vertex] at (\xx+2.5+0.5-0.1,1+2*\ext) {};

  \draw (\xx+0.75,1+4*\ext) node [anchor=south] {$-1$};
  \draw (\xx+2.5+0.7,0.65) node {$-d$};
\end{tikzpicture}

\caption{The iff gadget and its symbolic notation}
\label{figure-iff}
\end{figure}

\noindent Please notice the $-d$ label near one of the external edges. It indicates that, when applying the iff gadget, the original weight of this edge has to be multiplied by $-d$. As above, we have $d=1$ in the case of determinant and $d=-1$ for permanent.

If a component of a cycle cover $c$ passes left to right through the gadget then, by its symmetry, there exists a different cycle cover $c'$ with $w(c')=-w(c)$, thus $c$ and $c'$ cancel out.

The remaining, easy to compute, nonzero signatures of the iff gadget are shown in {\bf Figure \ref{figure-iff-signature}}.

\def\xx{5.5}
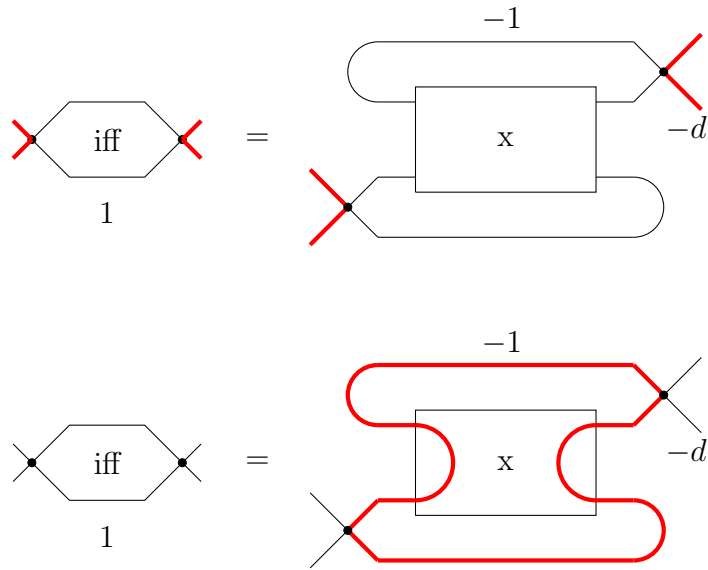
\begin{figure}[H]

\begin{tikzpicture}[scale=1]
  \def\yy{0}
  \iffgadget{0}{\yy}{t}{$1$}{$=$}
  \xgadget{\xx-0.2}{0}{n}{n}{n}{n}{0}{}{}

  \draw (\xx+2.5,0) arc (90:-90:2*\ext);
  \draw (\xx-0.5-2*\ext,-4*\ext) -- (\xx+2.5,-4*\ext);

  \draw (\xx-0.5-2*\ext,-4*\ext) -- ++(-2*\ext-0.5,2*\ext+0.5);
  \draw (\xx-0.5-2*\ext,0) -- ++(-2*\ext-0.5,-2*\ext-0.5);
  \draw[red, ultra thick]
    (\xx-0.5-2*\ext-2*\ext,2*\ext-4*\ext) -- ++(-0.5,0.5);
  \draw[red, ultra thick]
    (\xx-0.5-2*\ext-2*\ext,0-2*\ext) -- ++(-0.5,-0.5);
  \node[vertex] at (\xx-0.9-2*\ext,-2*\ext) {};

  \draw (\xx-0.5-2*\ext,1) arc (270:90:2*\ext);
  \draw (\xx-0.5-2*\ext,1+4*\ext) -- (\xx+2.5,1+4*\ext);

  \draw (\xx+2.5,1) -- ++(2*\ext+0.5,2*\ext+0.5);
  \draw (\xx+2.5,1+4*\ext) -- ++(2*\ext+0.5,-2*\ext-0.5);
  \draw[red, ultra thick]
    (\xx+2*\ext+2.5,2*\ext+1) -- ++(0.5,0.5);
  \draw[red, ultra thick]
    (\xx+2*\ext+2.5,1-2*\ext+4*\ext) -- ++(0.5,-0.5);
  \node[vertex] at (\xx+2.5+0.5-0.1,1+2*\ext) {};

  \draw (\xx+0.75,1+4*\ext) node [anchor=south] {$-1$};
  \draw (\xx+2.5+0.7,0.65) node {$-d$};

  \def\yy{-4.3}
  \iffgadget{0}{\yy}{n}{$1$}{$=$}
  \xgadget{\xx-0.2}{\yy}{t}{t}{t}{t}{2}{}{}

  \draw[red, ultra thick] (\xx+2.5,0+\yy) arc (90:-90:2*\ext);
  \draw[red, ultra thick] (\xx-0.5-2*\ext,-4*\ext+\yy) -- (\xx+2.5,-4*\ext+\yy);

  \draw[red, ultra thick] (\xx-0.5-2*\ext,-4*\ext+\yy) -- ++(-2*\ext,2*\ext);
  \draw[red, ultra thick] (\xx-0.5-2*\ext,0+\yy) -- ++(-2*\ext,-2*\ext);
  \draw
    (\xx-0.5-2*\ext-2*\ext,2*\ext-4*\ext+\yy) -- ++(-0.5,0.5);
  \draw
    (\xx-0.5-2*\ext-2*\ext,0-2*\ext+\yy) -- ++(-0.5,-0.5);
  \node[vertex] at (\xx-0.9-2*\ext,-2*\ext+\yy) {};

  \draw[red, ultra thick] (\xx-0.5-2*\ext,1+\yy) arc (270:90:2*\ext);
  \draw[red, ultra thick] (\xx-0.5-2*\ext,1+4*\ext+\yy) -- (\xx+2.5,1+4*\ext+\yy);

  \draw[red, ultra thick] (\xx+2.5,1+\yy) -- ++(2*\ext,2*\ext);
  \draw[red, ultra thick] (\xx+2.5,1+4*\ext+\yy) -- ++(2*\ext,-2*\ext);
  \draw
    (\xx+2*\ext+2.5,2*\ext+1+\yy) -- ++(0.5,0.5);
  \draw
    (\xx+2*\ext+2.5,1-2*\ext+4*\ext+\yy) -- ++(0.5,-0.5);
  \node[vertex] at (\xx+2.5+0.5-0.1,1+2*\ext+\yy) {};

  \draw (\xx+0.75,1+4*\ext+\yy) node [anchor=south] {$-1$};
  \draw (\xx+2.5+0.7,0.65+\yy) node {$-d$};
\end{tikzpicture}

\caption{The nonzero signatures of the iff gadget}
\label{figure-iff-signature}
\end{figure}

In the determinantal case, the inner loop introduces the $-1$ sign. This loop is not seen outside of the iff gadget hence its sign is included in the signature.

%%%%%%%%%%%%%%%%%%%%%%%%%%%%%%%%%%%%%%%%%%%%%%%%

In {\bf Figure \ref{figure-iff}}, we chose a construction that avoids loops and multiple edges. However, if we do not have to avoid multiple edges, we may use a simpler and more obvious construction of the iff gadget, shown in {\bf Figure \ref{figure-iff-multiple}} below.

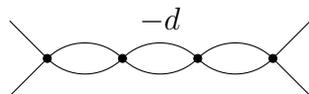
\begin{figure}[H]
\begin{tikzpicture}[scale=1]
  \tikzstyle{vertex}=[circle,minimum size=\nodesize,inner sep=0pt,draw=black,fill]

  \node [vertex] at (0,0) {};
  \node [vertex] at (1,0) {};
  \node [vertex] at (2,0) {};
  \node [vertex] at (3,0) {};

  \foreach \x in {0,1,2} {
    \draw (1+\x,0) arc (-45:-135:0.7071);
    \draw (0+\x,0) arc (135:45:0.7071); }

  \draw (0,0) --++(-0.5,0.5);
  \draw (0,0) --++(-0.5,-0.5);
  \draw (3,0) --++(0.5,0.5);
  \draw (3,0) --++(0.5,-0.5);

  \draw (1.5,1.414/2-0.5) node [anchor=south] {$-d$};
\end{tikzpicture}
\caption{A version of the iff gadget when multiple edges are allowed.}
\label{figure-iff-multiple}
\end{figure}

{\bf (3) The extended iff gadget.} We use the skew crossing gadgets and the iff gadget, defined above, to synchronize any two edges $e_1$ and $e_2$ in the graph. By the synchronization of the edges $e_1$ and $e_2$ we mean that an insertion of the {\em extended iff gadget} into the graph causes that determinant (resp. permanent) of the new graph to count precisely those cycle covers, of the original graph, which contain either both or none of the edges $e_1$ and $e_2$. Otherwise, the gadget has no effect on the remaining cycle covers. In {\bf Figures \ref{figure-extended-iff-1} and \ref{figure-extended-iff-2}} we present the construction of the extended iff gadget in the case when it goes across two other edges.

The remaining cases, when the number of edges $r_i$, $i=1,2,\ldots,r_n$, the gadget passes through, is different than $2$ and different configurations of the $r_i$'s that belong or not to the cycle cover, are analogous.

{\bf Figure \ref{figure-extended-iff-symbol}} shows the symbolic notation of the gadget.

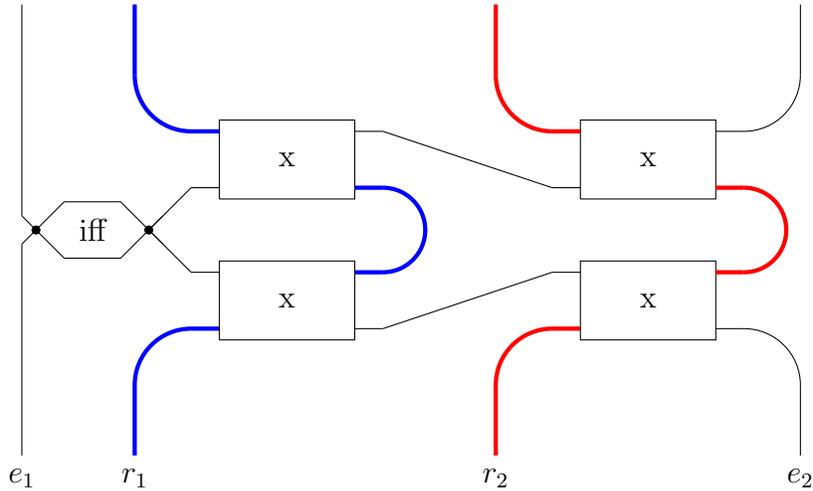
\begin{figure}[H]
\begin{tikzpicture}[scale=0.75,photon/.style={decorate,decoration={snake,post length=1mm}}]
  \tikzstyle{vertex}=[circle,minimum size=\nodesize,inner sep=0pt,draw=black,fill]
  \tikzstyle{vertexempty}=[circle,minimum size=\nodesize*2,inner sep=0pt,draw=black]

  \draw (0,0.25)-- ++(0,3.75);
  \draw (0,-0.25)-- ++(0,-3.75);
  \iffgadget{0.25}{-0.5}{n}{}{}

  \draw (2.25,0) -- ++(0.75,0.75);
  \draw[blue, ultra thick] (3,1.75) arc (270:180:1);
  \draw[blue, ultra thick] (2,2.75) -- ++(0,1.25);
  \xgadget{3.7}{0.75}{n}{n}{n}{n}{0}{}{}
  \draw (2.25,0) -- ++(0.75,-0.75);
  \draw[blue, ultra thick] (3,-1.75) arc (90:180:1);
  \draw[blue, ultra thick] (2,-2.75) -- ++(0,-1.25);
  \xgadget{3.7}{-1.75}{n}{n}{n}{n}{0}{}{}
  \draw[blue, ultra thick] (6.4,0.75) arc (90:-90:0.75);

  \draw[blue, ultra thick] (3,1.75)--++(0.5,0);
  \draw[blue, ultra thick] (3,-1.75)--++(0.5,0);
  \draw[blue, ultra thick] (5.9,0.75)--++(0.5,0);
  \draw[blue, ultra thick] (5.9,-0.75)--++(0.5,0);

  \draw[red, ultra thick] (9.4,1.75) arc (-90:-180:1);
  \draw[red, ultra thick] (8.4,2.75) -- ++(0,1.25);
  \xgadget{10.1}{0.75}{n}{n}{n}{n}{0}{}{}
  \draw (6.4,1.75) -- (9.4,0.75);
  \draw (6.4,-1.75) -- (9.4,-0.75);
  \draw[red, ultra thick] (9.4,-1.75) arc (90:180:1);
  \draw[red, ultra thick] (8.4,-2.75) -- ++(0,-1.25);
  \xgadget{10.1}{-1.75}{n}{n}{n}{n}{0}{}{}
  \draw[red, ultra thick] (12.8,0.75) arc (90:-90:0.75);

  \draw[red, ultra thick] (9.4,1.75)--++(0.5,0);
  \draw[red, ultra thick] (9.4,-1.75)--++(0.5,0);
  \draw[red, ultra thick] (12.3,0.75)--++(0.5,0);
  \draw[red, ultra thick] (12.3,-0.75)--++(0.5,0);

  \draw (12.8,1.75) arc (-90:0:1);
  \draw (13.8,2.75) -- ++(0,1.25);
  \draw (12.8,-1.75) arc (90:0:1);
  \draw (13.8,-2.75) -- ++(0,-1.25);

  \foreach \xx/\name in {0/e_1,2/r_1,8.4/r_2,13.8/e_2}
    \node at (\xx,-4.4) {$\name$};

\end{tikzpicture}
\caption{The extended iff gadget synchronizing two edges $e_1$ and $e_2$, across two other edges $r_1$ and $r_2$. The colors indicate the way the paths, representing $r_1$ and $r_2$, traverse the gadget in the case when both the $r_i$'s belong to the cycle cover but both the $e_i$'s do not.} \label{figure-extended-iff-1}
\end{figure}

\begin{figure}[H]
\begin{tikzpicture}[scale=0.75,photon/.style={decorate,decoration={snake,post length=1mm}}]
  \tikzstyle{vertex}=[circle,minimum size=\nodesize,inner sep=0pt,draw=black,fill]
  \tikzstyle{vertexempty}=[circle,minimum size=\nodesize*2,inner sep=0pt,draw=black]

  \draw[ultra thick] (0,0.25)-- ++(0,3.75);
  \draw[ultra thick] (0,-0.25)-- ++(0,-3.75);
  \draw[ultra thick] (0.25,0)-- ++(-0.25,0.25);
  \draw[ultra thick] (0.25,0)-- ++(-0.25,-0.25);
  \iffgadget{0.25}{-0.5}{n}{}{}
  \draw[blue, ultra thick] (2.25,0) -- ++(0.76,0.76);
  \draw[blue, ultra thick] (2.25,0) -- ++(0.76,-0.76);
  \node[vertex] at (2.25,0) {};

  \draw[blue, ultra thick] (3,1.75) arc (270:180:1);
  \draw[blue, ultra thick] (2,2.75) -- ++(0,1.25);
  \xgadget{3.7}{0.75}{n}{n}{n}{n}{2}{}{}
  \draw[blue, ultra thick] (3,-1.75) arc (90:180:1);
  \draw[blue, ultra thick] (2,-2.75) -- ++(0,-1.25);
  \xgadget{3.7}{-1.75}{n}{n}{n}{n}{2}{}{}
  \draw[blue, ultra thick] (3.5,1.75) arc (90:-90:0.5);
  \draw[blue, ultra thick] (3.5,-0.75) arc (90:-90:0.5);
  \draw[red, ultra thick] (6.4,0.75) arc (90:-90:0.75);

  \draw[blue, ultra thick] (3,1.75)--++(0.5,0);
  \draw[blue, ultra thick] (3,-1.75)--++(0.5,0);
  \draw[blue, ultra thick] (3,0.75)--++(0.5,0);
  \draw[blue, ultra thick] (3,-0.75)--++(0.5,0);

  \draw[red, ultra thick] (5.9,1.75) arc (90:270:0.5);
  \draw[red, ultra thick] (5.9,-0.75) arc (90:270:0.5);
  \draw[red, ultra thick] (5.9,1.75)--++(0.5,0);
  \draw[red, ultra thick] (5.9,-1.75)--++(0.5,0);
  \draw[red, ultra thick] (5.9,0.75)--++(0.5,0);
  \draw[red, ultra thick] (5.9,-0.75)--++(0.5,0);

  \draw[red, ultra thick] (9.4,1.75) arc (-90:-180:1);
  \draw[red, ultra thick] (8.4,2.75) -- ++(0,1.25);
  \xgadget{10.1}{0.75}{n}{n}{n}{n}{2}{}{}
  \draw[red, ultra thick] (6.4,1.75) -- (9.4,0.75);
  \draw[red, ultra thick] (6.4,-1.75) -- (9.4,-0.75);
  \draw[red, ultra thick] (9.4,-1.75) arc (90:180:1);
  \draw[red, ultra thick] (8.4,-2.75) -- ++(0,-1.25);
  \xgadget{10.1}{-1.75}{n}{n}{n}{n}{2}{}{}
  \draw[ultra thick] (12.8,0.75) arc (90:-90:0.75);

  \draw[red, ultra thick] (9.9,1.75) arc (90:-90:0.5);
  \draw[red, ultra thick] (9.9,-0.75) arc (90:-90:0.5);

  \draw[red, ultra thick] (9.4,1.75)--++(0.5,0);
  \draw[red, ultra thick] (9.4,-1.75)--++(0.5,0);
  \draw[red, ultra thick] (9.4,0.75)--++(0.5,0);
  \draw[red, ultra thick] (9.4,-0.75)--++(0.5,0);

  \draw[ultra thick] (12.3,1.75)--++(0.5,0);
  \draw[ultra thick] (12.3,-1.75)--++(0.5,0);
  \draw[ultra thick] (12.3,0.75)--++(0.5,0);
  \draw[ultra thick] (12.3,-0.75)--++(0.5,0);
  \draw[ultra thick] (12.3,1.75) arc (90:270:0.5);
  \draw[ultra thick] (12.3,-0.75) arc (90:270:0.5);
  \draw[ultra thick] (12.8,1.75) arc (-90:0:1);
  \draw[ultra thick] (13.8,2.75) -- ++(0,1.25);
  \draw[ultra thick] (12.8,-1.75) arc (90:0:1);
  \draw[ultra thick] (13.8,-2.75) -- ++(0,-1.25);

  \foreach \xx/\name in {0/e_1,2/r_1,8.4/r_2,13.8/e_2}
    \node at (\xx,-4.4) {$\name$};

\end{tikzpicture}
\caption{The same gadget as in Figure \ref{figure-extended-iff-1}. The case when all the $r_i$'s and the $e_i$'s belong to the cycle cover.}
\label{figure-extended-iff-2}
\end{figure}
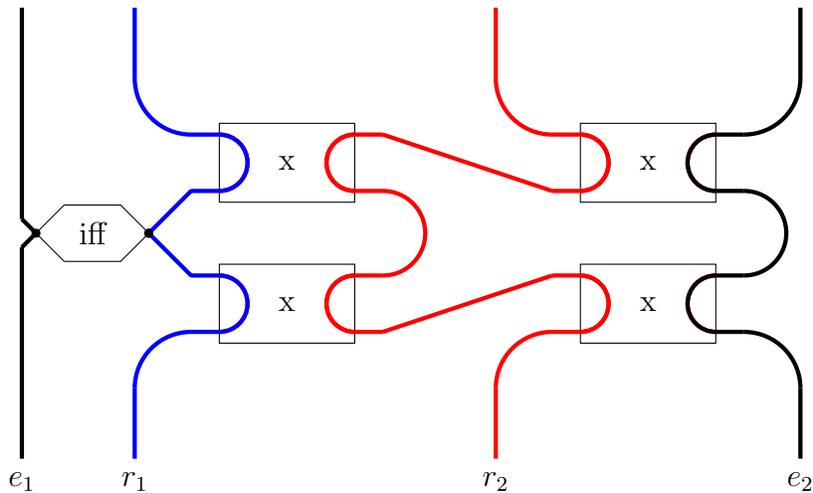

\begin{figure}[H]
\begin{tikzpicture}[scale=0.75,photon/.style={decorate,decoration={snake,post length=1mm}}]
  \tikzstyle{vertex}=[circle,minimum size=\nodesize,inner sep=0pt,draw=black,fill]
  \tikzstyle{vertexempty}=[circle,minimum size=\nodesize*2,inner sep=0pt,draw=black]

  \foreach \xx in {0,1,2,3}
    \draw (\xx,1)--(\xx,-1);
  \draw[blue, ultra thick] (1,1)--(1,-1);
  \draw[red, ultra thick] (2,1)--(2,-1);

  \node[vertexempty] at (0,0) {};
  \node[vertexempty] at (3,0) {};
  \draw[-,photon] (0,0)--(3,0);
  \foreach \xx/\name in {0/e_1,1/r_1,2/r_2,3/e_2}
    \node at (\xx,-1.4) {$\name$};

\end{tikzpicture}
\caption{A symbolic notation for the gadget shown in {\bf Figures \ref{figure-extended-iff-1} and \ref{figure-extended-iff-2}}.} \label{figure-extended-iff-symbol}
\end{figure}
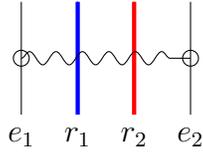

We see that, in {\bf Figures \ref{figure-extended-iff-1} and \ref{figure-extended-iff-2}}, every two skew crossing gadgets that are drawn one above the other, have both the same signature, either $1$ or $-1$. Therefore the crossing gadgets do not contribute to the values of the signature of the extended iff gadget. Neither does the iff gadget.

It is straightforward to see that those cycle covers, of the modified graph, that contribute to either determinant or permanent must contain either both the edges $e_1$ and $e_1$ or none of them.

It remains to notice that insertion of the extended iff gadget affects neither the number of cycle covers nor parity of the number of cycles in a cycle cover. The only restriction is that $e_1$ belongs to a cover if and only if $e_2$ does:

When neither $e_1$ nor $e_2$ is in the cover, a case similar to the one shown in {\bf Figure \ref{figure-extended-iff-1}}, whether a path representing $r_i$ belongs to the cover or not, does not affect other paths.

When both $e_1$ and $e_2$ are in the cover, we start with a case similar to the one shown in {\bf Figure \ref{figure-extended-iff-2}}, where the cycle cover passes through all the $r_i$'s, $i=1,2,\ldots,n$.  Removal of any path, representing an edge $r_i$, from the cycle cover, causes its central part to be filled by the next path on the right.

We conclude that whenever the extended iff gadget is placed between any two edges $e_1$ and $e_1$ as in {\bf Figure \ref{figure-extended-iff-symbol}}, the determinant as well as the permanent of the resulting graph counts exactly those cycle covers, of the original graph, which contain either both the $e_i$'s or none of them.

{\bf (4) The variable setting gadget.} For every variable $x_i$ in the formula $\phi$ we construct a gadget, as shown in {\bf Figure \ref{figure-variable}}.

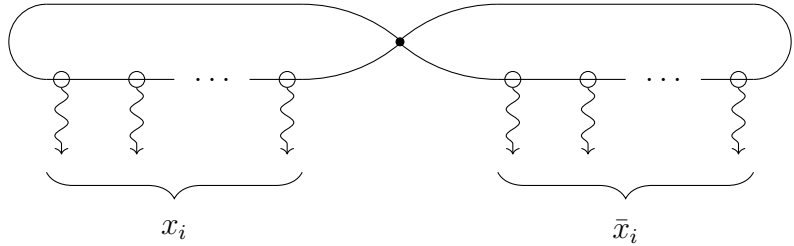
\begin{figure}[H]
\def\a{0.2}
\begin{tikzpicture}[scale=1,photon/.style={decorate,decoration={snake,post length=1mm}}]
  \tikzstyle{vertex}=[circle,minimum size=\nodesize,inner sep=0pt,draw=black,fill]
  \tikzstyle{vertexempty}=[circle,minimum size=\nodesize*2,inner sep=0pt,draw=black]

  \foreach \xx in {0,1,3,6,7,9}
  {
    \node[vertexempty] (a) at (\xx,0) {};
    \draw[->,photon] (a) -- ++(0,-1);
  }

  \draw[-] (-0.2,0) arc (-90:-270:0.5);
  \draw[-] (-0.2,1) -- ++ (3.4,0);
  \draw[-] (-0.2,0) -- (1.5,0);
  \node at (2,0) {$\ldots$};
  \draw[-] (2.5,0) -- (3.2,0);
  \draw[-] (3.2,1) arc (90:49:2);
  \draw[-] (3.2,0) arc (-90:-49:2);

  \node[vertex] at (4.5,0.5) {};

  \draw[-] (5.8,1) arc (90:131:2);
  \draw[-] (5.8,0) arc (-90:-131:2);
  \draw[-] (5.8,1) -- ++ (3.4,0);
  \draw[-] (9.2,1) arc (90:-90:0.5);
  \draw[-] (5.8,0) -- ++(1.7,0);
  \node at (8,0) {$\ldots$};
  \draw[-] (8.5,0) -- ++(0.7,0);

  \coordinate (a) at (-0.2,-1.2);
  \coordinate (b) at (3.2,-1.2);
  \draw[decorate,decoration={brace,amplitude=10pt,raise=1pt,mirror},yshift=0pt] (a) -- (b) node [midway,yshift=-15pt][anchor=north]{$x_i$};
  \coordinate (a) at (5.8,-1.2);
  \coordinate (b) at (9.2,-1.2);
  \draw[decorate,decoration={brace,amplitude=10pt,raise=1pt,mirror},yshift=0pt] (a) -- (b) node [midway,yshift=-15pt][anchor=north]{$\bar x_i$};

\end{tikzpicture}
\caption{The variable setting gadget}
\label{figure-variable}
\end{figure}

The gadget consists of two loops connected to a single vertex. Clearly every cycle cover contains exactly one of the two loops and the signature of the gadget is always $1$ so that addition of the variable setting gadget does not affect the determinant or permanent of the graph.
One of these two loops represents the variable $x_i$, while the other represents its negation $\bar x_i$.
The wavy lines in {\bf Figure \ref{figure-variable}} denote the extended iff gadgets which synchronize a loop with those edges in the clause gadgets, shown in {\bf Figure \ref{figure-clause-gadget}} below, which represent the same $x_i$ or $\bar x_i$, respectively. We have one extended iff gadget per one occurrence of $x_i$ or $\bar x_i$ in the formula $\phi$.

If $x_i$ or $\bar x_i$ is not present in the formula $\phi$ then the graph we obtain has a loop. If this is undesirable we may synchronize this loop with itself, using the iff gadget, which results in a loopless graph.

{\bf (5) The clause gadget.} For every clause of the form $a\vee b\vee c$ in the formula $\phi$ we construct a gadget, as shown in {\bf Figure \ref{figure-clause-gadget}}.

\newcommand\clausegadget[9]{
  \node[vertex] (v1) at (0,0) {};
  \node[vertex] (v2) at (1,1.5) {};
  \node[vertex] (v3) at (1,0.5) {};
  \node[vertex] (v4) at (1,-0.5) {};
  \node[vertex] (v5) at (3,1.5) {};
  \node[vertex] (v6) at (3,0.5) {};
  \node[vertex] (v7) at (3,-1.5) {};
  \node[vertex] (v8) at (4,0) {};
  \node[vertexempty] (a) at (1,-1) {};
  \node[vertexempty] (b) at (2,0.5) {};
  \node[vertexempty] (c) at (3,-2) {};

  \draw\ifnum\value{#2}=1[red, ultra thick]\fi (v2) -- ++(2,0);
  \draw\ifnum\value{#7}=1[red, ultra thick]\fi (v3) -- ++(2,0);
  \draw\ifnum\value{#4}=1[red, ultra thick]\fi (v2) -- ++(0,-1);
  \draw\ifnum\value{#5}=1[red, ultra thick]\fi (v5) -- ++(0,-1);

  \draw\ifnum\value{#1}=1[red, ultra thick]\fi (v2) arc (90:180:1);
  \draw\ifnum\value{#1}=1[red, ultra thick]\fi (v1) -- ++(0,0.5);
  \draw\ifnum\value{#6}=1[red, ultra thick]\fi (v3) arc (90:145:1.25);

  \draw\ifnum\value{#3}=1[red, ultra thick]\fi (v5) arc (90:0:1);
  \draw\ifnum\value{#3}=1[red, ultra thick]\fi (v8)--++(0,0.5);
  \draw\ifnum\value{#8}=1[red, ultra thick]\fi (v6) arc (90:35:1.25);

  \draw (v4) -- ++(2,0);
  \draw (v4) arc (-90:-145:1.25);
  \draw (3,-0.5) arc (-90:-35:1.25);
  \draw (c) arc (-90:270:0.25);

  \draw (v7) -- ++(-2,0);
  \draw (v7) arc (-90:0:1);
  \draw (1,-1.5) arc (-90:-180:1);
  \draw (v1) -- ++(0,-0.5);
  \draw (v8) -- ++(0,-0.5);
  \draw (a) arc (-90:270:0.25);

  \ifnum\value{#9}=1
    \draw[red, ultra thick] (v4) -- ++(2,0);
    \draw[red, ultra thick] (v4) arc (-90:-145:1.25);
    \draw[red, ultra thick] (3,-0.5) arc (-90:-35:1.25);
    \draw[red, ultra thick] (c) arc (-90:270:0.25);
  \fi
  \ifnum\value{#9}=2
    \draw[red, ultra thick] (v7) -- ++(-2,0);
    \draw[red, ultra thick] (v7) arc (-90:0:1);
    \draw[red, ultra thick] (1,-1.5) arc (-90:-180:1);
    \draw[red, ultra thick] (v1) -- ++(0,-0.5);
    \draw[red, ultra thick] (v8) -- ++(0,-0.5);
    \draw[red, ultra thick] (a) arc (-90:270:0.25);
  \fi
  \ifnum\value{#9}=3
    \draw[red, ultra thick] (v4) -- ++(2,0);
    \draw[red, ultra thick] (v4) arc (-90:-145:1.25);
    \draw[red, ultra thick] (3,-0.5) arc (-90:-35:1.25);

    \draw[red, ultra thick] (v7) -- ++(-2,0);
    \draw[red, ultra thick] (v7) arc (-90:0:1);
    \draw[red, ultra thick] (1,-1.5) arc (-90:-180:1);
    \draw[red, ultra thick] (v1) -- ++(0,-0.5);
    \draw[red, ultra thick] (v8) -- ++(0,-0.5);
  \fi
  \ifnum\value{#9}=-1
    \draw[red, ultra thick] (a) arc (-90:270:0.25);
    \draw[red, ultra thick] (c) arc (-90:270:0.25);
  \fi

  \draw[->,photon] (a) -- ++(0,-2);
  \draw[->,photon] (b) -- ++(0,-3.5);
  \draw[->,photon] (c) -- ++(0,-1);

  \node[vertex] (v1) at (v1) {};
  \node[vertex] (v2) at (v2) {};
  \node[vertex] (v3) at (v3) {};
  \node[vertex] (v4) at (v4) {};
  \node[vertex] (v5) at (v5) {};
  \node[vertex] (v6) at (v6) {};
  \node[vertex] (v7) at (v7) {};
  \node[vertex] (v8) at (v8) {};

}

\begin{figure}[H]
\begin{tikzpicture}[scale=1,photon/.style={decorate,decoration={snake,post length=1mm}}]
  \tikzstyle{vertex}=[circle,minimum size=\nodesize,inner sep=0pt,draw=black,fill]
  \tikzstyle{vertexempty}=[circle,minimum size=\nodesize*2,inner sep=0pt,draw=black]

  \clausegadget{n}{n}{n}{n}{n}{n}{n}{n}{n}
  \node at (3,1) [anchor=east] {$-d$};

  \node at (1,-3) [anchor=north] {$a$};
  \node at (2,-3) [anchor=north] {$b$};
  \node at (3,-3) [anchor=north] {$c$};
\end{tikzpicture}
\caption{The clause gadget encoding $a\vee b\vee c$} \label{figure-clause-gadget}
\end{figure}
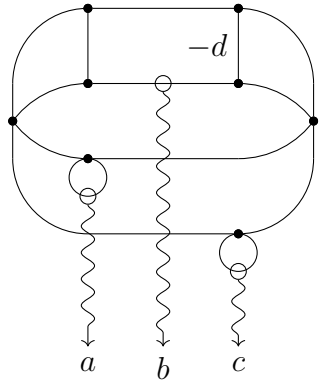

Three edges of the gadget represent the three literals $a,b,c\in\{x_1,\bar x_1,x_2,\bar x_2,\ldots\}$ in the clause. The wavy lines indicate the extended iff gadgets which synchronize these three edges with those loops in the variable setting gadgets which represent the same literal.

All the possible cycle covers of the clause gadget are shown in {\bf Figure \ref{figure-clause-gadget-covers}}. We see that the signature is $0$ if all the three literals are set to {\em false}, represented by $0$. Otherwise, all the signatures are equal $1$ in the permanental case and $-1$ in the determinantal case.

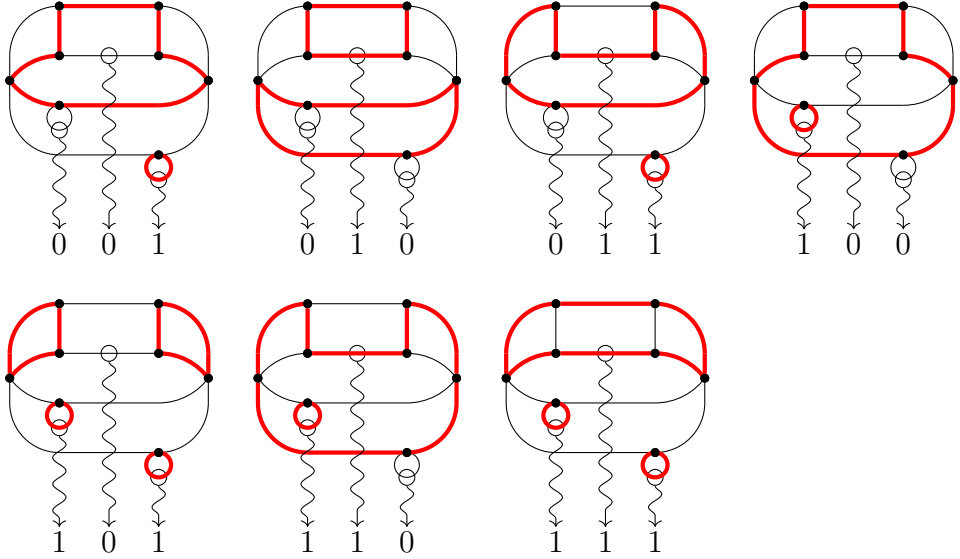
\begin{figure}[H]
\begin{tikzpicture}[scale=0.66,photon/.style={decorate,decoration={snake,post length=1mm}}]
  \tikzstyle{vertex}=[circle,minimum size=\nodesize,inner sep=0pt,draw=black,fill]
  \tikzstyle{vertexempty}=[circle,minimum size=\nodesize*2,inner sep=0pt,draw=black]

  \begin{scope}
    \clausegadget{n}{t}{n}{t}{t}{t}{n}{t}{t}
    \foreach \xx/\name in {1/0,2/0,3/1}
    \node at (\xx,-3.3) {$\name$};
  \end{scope}
  \begin{scope}[shift={(5,0)}]
    \clausegadget{n}{t}{n}{t}{t}{n}{t}{n}{dt}
    \foreach \xx/\name in {1/0,2/1,3/0}
    \node at (\xx,-3.3) {$\name$};
  \end{scope}
  \begin{scope}[shift={(10,0)}]
    \clausegadget{t}{n}{t}{t}{t}{n}{t}{n}{t}
    \foreach \xx/\name in {1/0,2/1,3/1}
    \node at (\xx,-3.3) {$\name$};
  \end{scope}
  \begin{scope}[shift={(15,0)}]
    \clausegadget{n}{t}{n}{t}{t}{t}{n}{t}{d}
    \foreach \xx/\name in {1/1,2/0,3/0}
    \node at (\xx,-3.3) {$\name$};
  \end{scope}
  \begin{scope}[shift={(0,-6)}]
    \clausegadget{t}{n}{t}{t}{t}{t}{n}{t}{nn}
    \foreach \xx/\name in {1/1,2/0,3/1}
    \node at (\xx,-3.3) {$\name$};
  \end{scope}
  \begin{scope}[shift={(5,-6)}]
    \clausegadget{t}{n}{t}{t}{t}{n}{t}{n}{d}
    \foreach \xx/\name in {1/1,2/1,3/0}
    \node at (\xx,-3.3) {$\name$};
  \end{scope}
  \begin{scope}[shift={(10,-6)}]
    \clausegadget{t}{t}{t}{n}{n}{t}{t}{t}{nn}
    \foreach \xx/\name in {1/1,2/1,3/1}
    \node at (\xx,-3.3) {$\name$};
  \end{scope}

\end{tikzpicture}
\caption{All the $7$ cycle covers of the clause gadget. The labels indicate the corresponding valuations of $a$, $b$, $c$.}
\label{figure-clause-gadget-covers}
\end{figure}

The construction of this gadget was chosen so as to make the exposition more straightforward, however, the gadget can be simplified by removing those edges that correspond to literals $a$ and $c$ and connecting the respective extended iff gadgets to $\bar a$ and $\bar c$.

\begin{theorem}
\label{theorem-degree-4}
  The undirected determinant and the undirected permanent of planar graphs, whose vertices have degrees $3$ or $4$, are \#P-complete.
\end{theorem}

\begin{proof}
  Consider a Boolean formula $\phi(x_1,x_2,\ldots,x_n)=c_1\wedge c_2\wedge\ldots\wedge c_m$, where $c_j=t_{j,1}\vee t_{j,2}\vee t_{j,3}$ with $t_{j,k}\in\{x_1,\bar x_1,x_2,\bar x_2,\ldots,x_n,\bar x_n\}$, where $j=1,2,\ldots,m$ and $k=1,2,3$. The $\bar x_i$ represents the negation of $x_i$.

  We construct the planar graphs $A_\phi$ for determinant and $B_\phi$ for permanent by taking one clause gadget for every $c_i$ and one variable setting gadget for every $x_i$. For every occurrence of $x_i$ or $\bar x_i$ in the formula, we use the extended iff gadget to synchronize the suitable edge of the variable setting gadget with the corresponding edge in the clause gadget.

  The properties of the gadgets, described in this section, imply that the number of satisfying valuations of $\phi$ is equal to $(-1)^m$ times the undirected determinant of $A_\phi$ and to the undirected permanent of $B_\phi$. This proves the theorem.
\end{proof}

\section{The undirected permanent of planar graphs of maximum degree $3$ reduces to the FKT algorithm}
\label{section-permanent-3}

Let $G$ be a weighted undirected planar graph of maximum degree $3$.
If $G$ contains a vertex of degree $0$ or $1$ then clearly
$$
  \uperm G = 0
$$
Suppose that all the vertices of $G$ have degree $2$ or $3$ and all the edges have nonzero weight. We define $G_{inv}$ as the subgraph of $G$ induced by vertices of degree $3$, and the weights of $G_{inv}$ are inverted, that is: $w_{G_{inv}}(e)=1/w_G(e)$, where $e$ denotes an edge and $w_G(e)$ denotes its weight in the graph $G$. The complement of a cycle cover of $G$ is a perfect matching in $G_{inv}$. In fact the complement, in the set of edges of $G$, establishes a one-to-one correspondence between the perfect matchings in $G_{inv}$ and the cycle covers of $G$.
The product of the weight of a cycle cover and the corresponding perfect matching is always equal to the product of weights of all the edges of $G$, denote it by $p$.
Since the weights in $G_{inv}$ are inverted, we see that
$$
  \uperm G = p\cdot\perfmatch G_{inv}
$$
Since $G_{inv}$ is planar, the sum of the weighted perfect matching in $G_{inv}$ is computed in polynomial time by the FKT algorithm.

\section{The undirected determinant of cubic planar graphs is \#P-complete}
\label{section-determinant-3}

Unlike the undirected permanent, it turns out that the undirected determinant is \#P-complete even in the case of cubic planar graphs. By Theorem \ref{theorem-degree-4} we already know that the undirected determinant is \#P-complete in the case of planar graphs whose vertices have degree $3$ or $4$, hence it is enough to construct a cubic planar gadget, whose signature is the same as that of a single vertex of degree $4$.

Let us draw attention to the fact that while all the gadgets constructed in Section \ref{section-degree-4} had edges of weight either $1$ or $-1$, here we need to know that $2$ is invertible or, at least, it is a non-zero-divisor.
This restriction is impossible to avoid since, modulo $2$, the undirected determinant coincides with the undirected permanent, and the latter reduces to the FKT algorithm, as proved in Section \ref{section-permanent-3}.

{\bf (1) The auxiliary gadget} is shown in {\bf Figure \ref{figure-auxiliary-gadget}}.

\newcommand\ugadgetoneshape{
  \node[vertex] (v1) at (0,0) {};
  \node[vertex] (v2) at (1,0) {};
  \node[vertex] (v3) at (2,0) {};
  \node[vertex] (v4) at (2,1) {};
  \node[vertex] (v5) at (2,2) {};
  \node[vertex] (v6) at (1,2) {};
  \node[vertex] (v7) at (0,2) {};
  \node[vertex] (v8) at (0,1) {};

  \node[vertex] (w1) at (1,0.5) {};
  \node[vertex] (w2) at (1.5,1) {};
  \node[vertex] (w3) at (1,1.5) {};
  \node[vertex] (w4) at (0.5,1) {};

  \draw (v1) -- (v3) -- (v5) -- (v7) -- (v1);
  \draw (v2) -- (w1);
  \draw (v4) -- (w2);
  \draw (v6) -- (w3);
  \draw (v8) -- (w4);
  \draw (w1) -- (w2) -- (w3) -- (w4) -- (w1);
  \draw (v7) -- ++(-0.5,0.5);
  \draw (v5) -- ++(0.5,0.5);
  \draw (v1) -- ++(-0.5,-0.5);
  \draw (v3) -- ++(0.5,-0.5);
}

\newcommand\ugadgetone[6]{ % #1-inner, #2#3-empty boudary, #4#5-outgoing
  \node[vertex] (v1) at (0,0) {};
  \node[vertex] (v2) at (1,0) {};
  \node[vertex] (v3) at (2,0) {};
  \node[vertex] (v4) at (2,1) {};
  \node[vertex] (v5) at (2,2) {};
  \node[vertex] (v6) at (1,2) {};
  \node[vertex] (v7) at (0,2) {};
  \node[vertex] (v8) at (0,1) {};

  \node[vertex] (w1) at (1,0.5) {};
  \node[vertex] (w2) at (1.5,1) {};
  \node[vertex] (w3) at (1,1.5) {};
  \node[vertex] (w4) at (0.5,1) {};

  \ifcase#4
    \draw[red, ultra thick] (v1) -- ++(-0.5,-0.5);
  \or
    \draw[red, ultra thick] (v3) -- ++(0.5,-0.5);
  \or
    \draw[red, ultra thick] (v5) -- ++(0.5,0.5);
  \or
    \draw[red, ultra thick] (v7) -- ++(-0.5,0.5);
  \fi
  \ifcase#5
    \draw[red, ultra thick] (v1) -- ++(-0.5,-0.5);
  \or
    \draw[red, ultra thick] (v3) -- ++(0.5,-0.5);
  \or
    \draw[red, ultra thick] (v5) -- ++(0.5,0.5);
  \or
    \draw[red, ultra thick] (v7) -- ++(-0.5,0.5);
  \fi

  \ifcase#1
    \draw[red, ultra thick] (w1) -- (w2) -- (w3) -- (w4);
    \draw[red, ultra thick] (w1) -- (v2);
    \draw[red, ultra thick] (w4) -- (v8);
    \or
    \draw[red, ultra thick] (w2) -- (w3) -- (w4) -- (w1);
    \draw[red, ultra thick] (w2) -- (v4);
    \draw[red, ultra thick] (w1) -- (v2);
    \or
    \draw[red, ultra thick] (w3) -- (w4) -- (w1) -- (w2);
    \draw[red, ultra thick] (w3) -- (v6);
    \draw[red, ultra thick] (w2) -- (v4);
    \or
    \draw[red, ultra thick] (w4) -- (w1) -- (w2) -- (w3);
    \draw[red, ultra thick] (w4) -- (v8);
    \draw[red, ultra thick] (w3) -- (v6);
  \fi
  \ifnum#2=9\else
    \ifnum#2=1\else\ifnum#3=1\else \draw[red, ultra thick] (v1) -- (v2); \fi\fi
    \ifnum#2=2\else\ifnum#3=2\else \draw[red, ultra thick] (v2) -- (v3); \fi\fi
    \ifnum#2=3\else\ifnum#3=3{}\else \draw[red, ultra thick] (v3) -- (v4); \fi\fi
    \ifnum#2=4\else\ifnum#3=4\else \draw[red, ultra thick] (v4) -- (v5); \fi\fi
    \ifnum#2=5\else\ifnum#3=5\else \draw[red, ultra thick] (v5) -- (v6); \fi\fi
    \ifnum#2=6\else\ifnum#3=6\else \draw[red, ultra thick] (v6) -- (v7); \fi\fi
    \ifnum#2=7\else\ifnum#3=7\else \draw[red, ultra thick] (v7) -- (v8); \fi\fi
    \ifnum#2=8\else\ifnum#3=8\else \draw[red, ultra thick] (v8) -- (v1); \fi\fi
  \fi

  \ifnum#6=1
    \draw[red, ultra thick] (w1) -- (v2) -- (v1) -- ++(-0.5,-0.5);
    \draw[red, ultra thick] (w2) -- (v4) -- (v3) -- ++(0.5,-0.5);
    \draw[red, ultra thick] (w3) -- (v6) -- (v5) -- ++(0.5,0.5);
    \draw[red, ultra thick] (w4) -- (v8) -- (v7) -- ++(-0.5,0.5);
    \draw[red, ultra thick] (w1) -- (w4) (w2) -- (w3);
  \fi
  \ifnum#6=2
    \draw[red, ultra thick] (w1) -- (v2) -- (v1) -- ++(-0.5,-0.5);
    \draw[red, ultra thick] (w2) -- (v4) -- (v3) -- ++(0.5,-0.5);
    \draw[red, ultra thick] (w3) -- (v6) -- (v5) -- ++(0.5,0.5);
    \draw[red, ultra thick] (w4) -- (v8) -- (v7) -- ++(-0.5,0.5);
    \draw[red, ultra thick] (w1) -- (w2) (w4) -- (w3);
  \fi
  \ifnum#6=3
    \draw[red, ultra thick] (w1) -- (w4) (w2) -- (w3);
    \draw[red, ultra thick] (w1) -- (w2) (w4) -- (w3);
  \fi
  \ifnum#6=4
    \draw[red, ultra thick] (w4) -- (v8) -- (v1) -- ++(-0.5,-0.5);
    \draw[red, ultra thick] (w1) -- (v2) -- (v3) -- ++(0.5,-0.5);
    \draw[red, ultra thick] (w2) -- (v4) -- (v5) -- ++(0.5,0.5);
    \draw[red, ultra thick] (w3) -- (v6) -- (v7) -- ++(-0.5,0.5);
    \draw[red, ultra thick] (w1) -- (w4) (w2) -- (w3);
  \fi
  \ifnum#6=5
    \draw[red, ultra thick] (w4) -- (v8) -- (v1) -- ++(-0.5,-0.5);
    \draw[red, ultra thick] (w1) -- (v2) -- (v3) -- ++(0.5,-0.5);
    \draw[red, ultra thick] (w2) -- (v4) -- (v5) -- ++(0.5,0.5);
    \draw[red, ultra thick] (w3) -- (v6) -- (v7) -- ++(-0.5,0.5);
    \draw[red, ultra thick] (w1) -- (w2) (w4) -- (w3);
  \fi
}

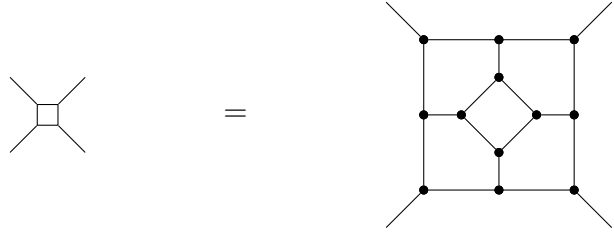
\begin{figure}[H]
\begin{tikzpicture}[scale=1,photon/.style={decorate,decoration={snake,post length=1mm}}]
  \tikzstyle{vertex}=[circle,minimum size=\nodesize,inner sep=0pt,draw=black,fill]
  \tikzstyle{vertexempty}=[circle,minimum size=\nodesize*2,inner sep=0pt,draw=black]

  \node[rectangle, draw] (d) at (-5,1) {};
  \draw (d) -- ++(-0.5,0.5);
  \draw (d) -- ++(0.5,0.5);
  \draw (d) -- ++(-0.5,-0.5);
  \draw (d) -- ++(0.5,-0.5);

  \node at (-2.5,1) {$=$};

  \begin{scope}
    \ugadgetoneshape
    \ugadgetone{9}{9}{9}{9}{9}{9}
  \end{scope}

\end{tikzpicture}
\caption{The auxiliary gadget and its symbolic notation.} \label{figure-auxiliary-gadget}
\end{figure}

The signatures, up to symmetry, corresponding to the different ways a cycle cover can traverse the auxiliary gadget, are shown in {\bf Figure \ref{figure-auxiliary-gadget-signature}}.
The last (i.e. the fourth) signature, is included only for the sake of completeness, as it is never realized in the subsequent constructions.
The value $s$, in the computation of this signature, is either $1$ or $-1$, depending on the choice which way of passing through the gadget is positive. We do not make this choice since in any case this signature is $0$.

\def\separate{3.5}
\begin{figure}[H]
\begin{tikzpicture}[scale=0.8,photon/.style={decorate,decoration={snake,post length=1mm}}]
  \tikzstyle{vertex}=[circle,minimum size=\nodesize,inner sep=0pt,draw=black,fill]
  \tikzstyle{vertexempty}=[circle,minimum size=\nodesize*2,inner sep=0pt,draw=black]

  \begin{scope}[shift={(0,0)}]
    \node[rectangle, draw] (d) at (0,0) {};
    \draw (d) -- ++(-0.5,0.5);
    \draw (d) -- ++(0.5,0.5);
    \draw (d) -- ++(-0.5,-0.5);
    \draw (d) -- ++(0.5,-0.5);

    \node at (1.5,0) {$=$};
    \node at (0,-1) {$1$};

    \begin{scope}[scale=0.75, shift={(4,-1)}]
      \ugadgetoneshape
      \ugadgetone{9}{1}{2}{9}{9}{9}
      \ugadgetone{9}{3}{4}{9}{9}{3}
      \node at (1,-1) {$1$};
    \end{scope}
  \end{scope}

  \begin{scope}[shift={(0,-\separate)}]
    \node[rectangle, draw] (d) at (0,0) {};
    \draw[red, ultra thick] (d) -- ++(-0.5,0.5);
    \draw (d) -- ++(0.5,0.5);
    \draw[red, ultra thick] (d) -- ++(-0.5,-0.5);
    \draw (d) -- ++(0.5,-0.5);

    \node at (1.5,0) {$=$};
    \node at (0,-1) {$2$};

    \begin{scope}[scale=0.75, shift={(4,-1)}]
      \ugadgetoneshape
      \ugadgetone{3}{6}{8}{0}{3}{9}
      \node at (1,-1) {$1$};
    \end{scope}
    \node at (6,0) {$+$};
    \begin{scope}[scale=0.75, shift={(10,-1)}]
      \ugadgetoneshape
      \ugadgetone{0}{1}{7}{0}{3}{9}
      \node at (1,-1) {$1$};
    \end{scope}
  \end{scope}

  \begin{scope}[shift={(0,-2*\separate)}]
    \node[rectangle, draw] (d) at (0,0) {};
    \draw (d) -- ++(-0.5,0.5);
    \draw[red, ultra thick] (d) -- ++(0.5,0.5);
    \draw[red, ultra thick] (d) -- ++(-0.5,-0.5);
    \draw (d) -- ++(0.5,-0.5);

    \node at (1.5,0) {$=$};
    \node at (0,-1) {$-2$};

    \begin{scope}[scale=0.75, shift={(4,-1)}]
      \ugadgetoneshape
      \ugadgetone{3}{5}{8}{0}{2}{9}
      \node at (1,-1) {$-1$};
    \end{scope}
    \node at (6,0) {$+$};
    \begin{scope}[scale=0.75, shift={(10,-1)}]
      \ugadgetoneshape
      \ugadgetone{1}{1}{4}{0}{2}{9}
      \node at (1,-1) {$-1$};
    \end{scope}
  \end{scope}

  \begin{scope}[shift={(0,-3*\separate)}]
    \node[rectangle, draw] (d) at (0,0) {};
    \draw[red, ultra thick] (d) -- ++(-0.5,0.5);
    \draw[red, ultra thick] (d) -- ++(0.5,0.5);
    \draw[red, ultra thick] (d) -- ++(-0.5,-0.5);
    \draw[red, ultra thick] (d) -- ++(0.5,-0.5);

    \node at (1.5,0) {$=$};
    \node at (0,-1) {$0$};

    \begin{scope}[scale=0.75, shift={(4,-1)}]
      \ugadgetoneshape
      \ugadgetone{9}{9}{9}{9}{9}{2}
      \node at (1,-1) {$s$};
    \end{scope}
    \node at (6,0) {$+$};
    \begin{scope}[scale=0.75, shift={(10,-1)}]
      \ugadgetoneshape
      \ugadgetone{9}{9}{9}{9}{9}{4}
      \node at (1,-1) {$s$};
    \end{scope}
    \node at (10.5,0) {$+$};
    \begin{scope}[scale=0.75, shift={(16,-1)}]
      \ugadgetoneshape
      \ugadgetone{9}{9}{9}{9}{9}{1}
      \node at (1,-1) {$-s$};
    \end{scope}
    \node at (15,0) {$+$};
    \begin{scope}[scale=0.75, shift={(22,-1)}]
      \ugadgetoneshape
      \ugadgetone{9}{9}{9}{9}{9}{5}
      \node at (1,-1) {$-s$};
    \end{scope}
  \end{scope}

\end{tikzpicture}
\caption{The signatures, up to symmetry, of the auxiliary gadget.}
\label{figure-auxiliary-gadget-signature}
\end{figure}
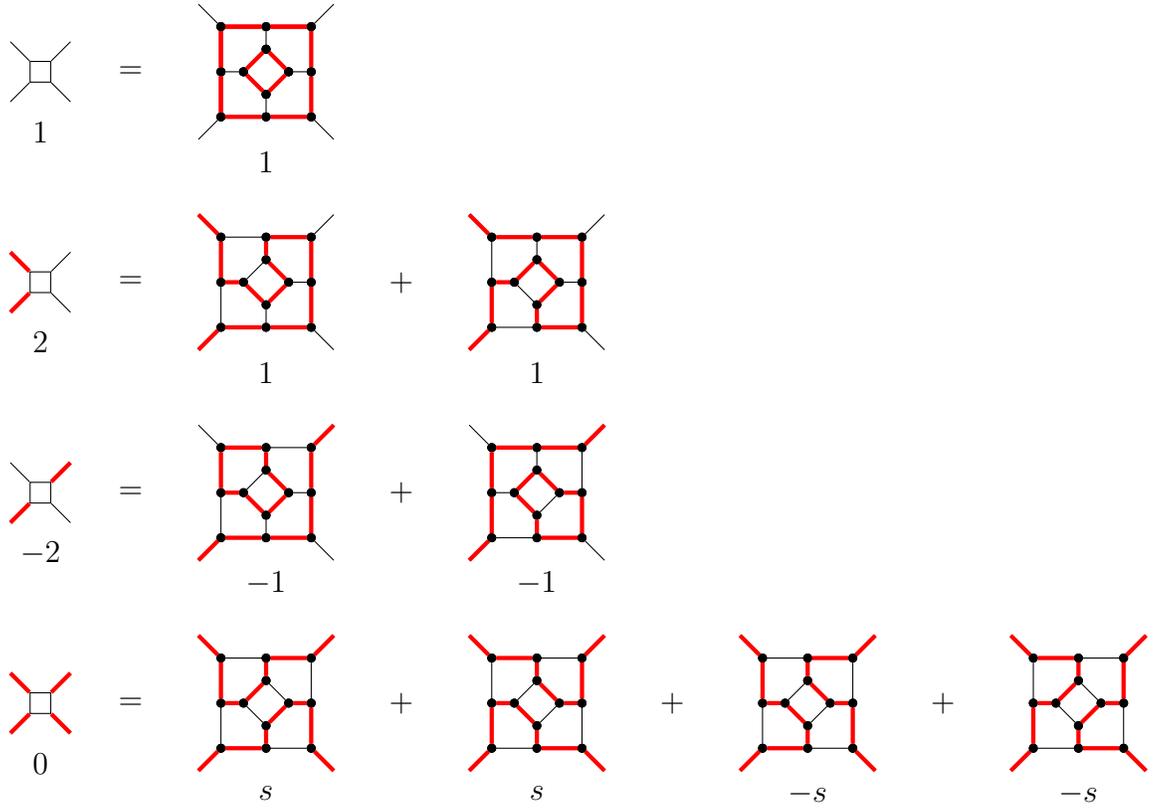

{\bf (2) The null edge gadget} is shown in {\bf Figure \ref{figure-null-edge}}. This gadget plays the role of an edge of weight $0$, so that the gadget is not necessary but convenient, if we want to assure that all vertices have degree $3$ and all the weights are invertible. The arguments, employed in Section \ref{section-semi-pfaffian}, are more neat when we deal with cubic graphs instead of graphs of degree at most $3$.

\def\separate{3}
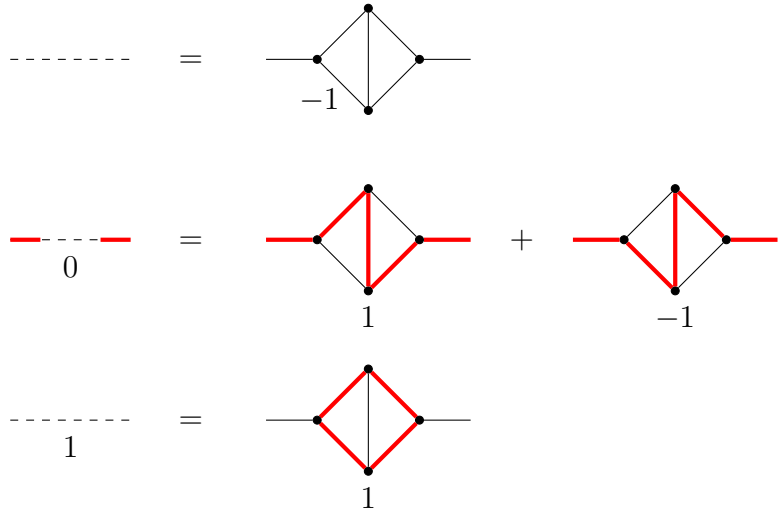
\begin{figure}[H]
\begin{tikzpicture}[scale=0.8,photon/.style={decorate,decoration={snake,post length=1mm}}]
  \tikzstyle{vertex}=[circle,minimum size=\nodesize,inner sep=0pt,draw=black,fill]
  \tikzstyle{vertexempty}=[circle,minimum size=\nodesize*2,inner sep=0pt,draw=black]

  \begin{scope}[shift={(0,0)}]
    \draw[dashed] (0,0) -- ++(2,0);

    \node at (3,-0.035) {$=$};
    \draw (1,0) node [anchor=north] {$ $};

    \begin{scope}[scale=0.85, shift={(5,0)}]
      \node[vertex] (v1) at (1,0) {};
      \node[vertex] (v2) at (2,1) {};
      \node[vertex] (v3) at (2,-1) {};
      \node[vertex] (v4) at (3,0) {};
      \draw (0,0) -- (v1) -- (v2) -- (v3) -- (v4) -- (4,0);
      \draw (v1) -- (v3) (v2) -- (v4);
      \draw (1.65,-0.35) node [anchor=north east] {$-1$};
    \end{scope}
  \end{scope}

  \begin{scope}[shift={(0,-\separate)}]
    \draw[dashed] (0,0) -- ++(2,0);
    \draw[red, ultra thick] (0,0) -- ++(0.5,0);
    \draw[red, ultra thick] (2,0) -- ++(-0.5,0);

    \node at (3,-0.035) {$=$};
    \draw (1,-0.1) node [anchor=north] {$0$};

    \begin{scope}[scale=0.85, shift={(5,0)}]
      \node[vertex] (v1) at (1,0) {};
      \node[vertex] (v2) at (2,1) {};
      \node[vertex] (v3) at (2,-1) {};
      \node[vertex] (v4) at (3,0) {};
      \draw[red, ultra thick] (0,0) -- (v1) -- (v2) -- (v3) -- (v4) -- (4,0);
      \draw (v1) -- (v3) (v2) -- (v4);
      \draw (2,-1.1) node [anchor=north] {$1$};
    \end{scope}
    \node at (8.5,0) {$+$};
    \begin{scope}[scale=0.85, shift={(11,0)}]
      \node[vertex] (v1) at (1,0) {};
      \node[vertex] (v2) at (2,1) {};
      \node[vertex] (v3) at (2,-1) {};
      \node[vertex] (v4) at (3,0) {};
      \draw[red, ultra thick] (0,0) -- (v1) -- (v3) -- (v2) -- (v4) -- (4,0);
      \draw (v1) -- (v2) (v3) -- (v4);
      \draw (2,-1.1) node [anchor=north] {$-1$};
    \end{scope}
  \end{scope}

  \begin{scope}[shift={(0,-2*\separate)}]
    \draw[dashed] (0,0) -- ++(2,0);

    \node at (3,-0.035) {$=$};
    \draw (1,-0.1) node [anchor=north] {$1$};

    \begin{scope}[scale=0.85, shift={(5,0)}]
      \node[vertex] (v1) at (1,0) {};
      \node[vertex] (v2) at (2,1) {};
      \node[vertex] (v3) at (2,-1) {};
      \node[vertex] (v4) at (3,0) {};
      \draw (0,0) -- (v1) -- (v2) -- (v3) -- (v4) -- (4,0);
      \draw (v1) -- (v3) (v2) -- (v4);
      \draw[red, ultra thick] (v1) -- (v2) -- (v4) -- (v3) -- (v1);
      \draw (2,-1.1) node [anchor=north] {$1$};
    \end{scope}
  \end{scope}
\end{tikzpicture}
\caption{The null edge gadget and its signature.}
\label{figure-null-edge}
\end{figure}

{\bf (3) The degree $4$ vertex gadget} is constructed in {\bf Figure \ref{figure-vertex-gadget}}.

\newcommand\ugadgetvert[1]{
  \node[rectangle, draw] (v1) at (0,0) {};
  \node[rectangle, draw] (v2) at (1,0) {};
  \node[rectangle, draw] (v3) at (1,1) {};
  \node[rectangle, draw] (v4) at (0,1) {};

  \draw (v1) to[out=-45,in=-135] (v2);
  \draw (v2) to[out=45,in=-45] (v3);
  \draw (v3) to[out=135,in=45] (v4);
  \draw (v4) to[out=225,in=135] (v1);

  \draw[dashed] (v1) to[out=45,in=-45] (v4);
  \draw[dashed] (v2) to[out=135,in=-135] (v3);

  \draw (v1) -- ++(-0.5,-0.5);
  \draw (v2) -- ++(0.5,-0.5);
  \draw (v3) -- ++(0.5,0.5);
  \draw (v4) -- ++(-0.5,0.5);

  \ifnum#1=1
    \draw (0.5,-0.6) node {$-{1\over 2}$};
    \draw (1.6,0.5) node {$-{1\over 2}$};
    \draw (0.5,1.6) node {$-{1\over 2}$};
    \draw (-0.6,0.5) node {$-{1\over 2}$};
  \fi
}

\newcommand\ugadgetvertcolors[3]{
  \ifcase#1
    \node[rectangle, red, draw, ultra thick] (v1) at (0,0) {};
  \or
    \node[rectangle, red, draw, ultra thick] (v2) at (1,0) {};
  \or
    \node[rectangle, red, draw, ultra thick] (v3) at (1,1) {};
  \or
    \node[rectangle, red, draw, ultra thick] (v4) at (0,1) {};
  \fi
  \ifcase#2
    \draw[red, draw, ultra thick] (v1) to[out=-45,in=-135] (v2);
  \or
    \draw[red, draw, ultra thick] (v2) to[out=45,in=-45] (v3);
  \or
    \draw[red, draw, ultra thick] (v3) to[out=135,in=45] (v4);
  \or
    \draw[red, draw, ultra thick] (v4) to[out=225,in=135] (v1);
  \fi
  \ifcase#3
    \draw[red, draw, ultra thick] (v1) -- ++(-0.5,-0.5);
  \or
    \draw[red, draw, ultra thick] (v2) -- ++(0.5,-0.5);
  \or
    \draw[red, draw, ultra thick] (v3) -- ++(0.5,0.5);
  \or
    \draw[red, draw, ultra thick] (v4) -- ++(-0.5,0.5);
  \fi
}

\begin{figure}[H]
\begin{tikzpicture}[scale=1,photon/.style={decorate,decoration={snake,post length=1mm}}]
  \tikzstyle{vertex}=[circle,minimum size=\nodesize,inner sep=0pt,draw=black,fill]
  \tikzstyle{vertexempty}=[circle,minimum size=\nodesize*2,inner sep=0pt,draw=black]

  \node[circle, draw] (d) at (-5,0.5) {};
  \draw (d) -- ++(-0.5,0.5);
  \draw (d) -- ++(0.5,0.5);
  \draw (d) -- ++(-0.5,-0.5);
  \draw (d) -- ++(0.5,-0.5);

  \node at (-2.5,0.5) {$=$};

  \begin{scope}
    \ugadgetvert{1}
  \end{scope}

\end{tikzpicture}
\caption{The degree $4$ vertex gadget, and its symbolic notation.}
\label{figure-vertex-gadget}
\end{figure}
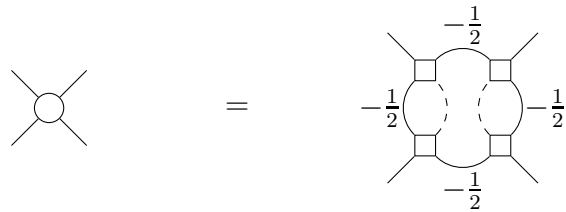

The signatures of the degree $4$ vertex gadget are listed, up to symmetry, in {\bf Figure \ref{figure-vertex-gadget-signature}}.

\def\separate{2.7}
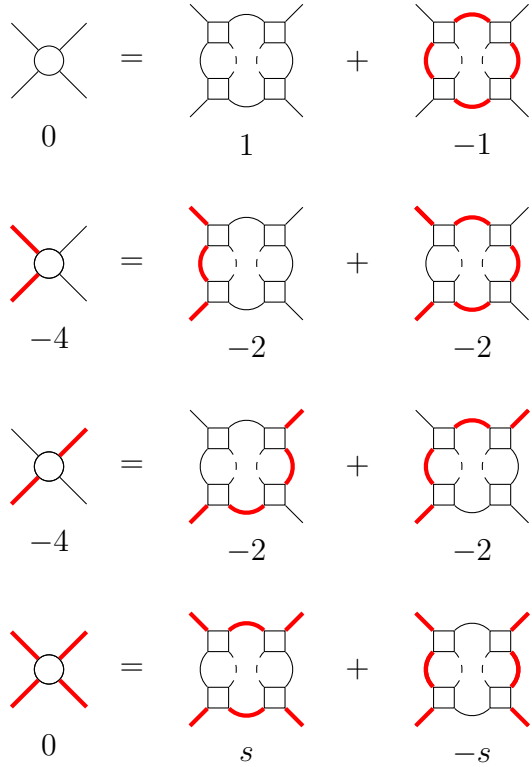
\begin{figure}[H]
\begin{tikzpicture}[scale=1,photon/.style={decorate,decoration={snake,post length=1mm}}]
  \tikzstyle{vertex}=[circle,minimum size=\nodesize,inner sep=0pt,draw=black,fill]
  \tikzstyle{vertexempty}=[circle,minimum size=\nodesize*2,inner sep=0pt,draw=black]

  \begin{scope}
    \node[circle, draw] (d) at (0,0) {};
    \draw (d) -- ++(-0.5,0.5);
    \draw (d) -- ++(0.5,0.5);
    \draw (d) -- ++(-0.5,-0.5);
    \draw (d) -- ++(0.5,-0.5);

    \node at (1.1,0) {$=$};
    \node at (0,-1) {$0$};

    \begin{scope}[scale=0.75, shift={(3,-0.5)}]
      \ugadgetvert{0}
      \node at (0.5,-1) {$1$};
    \end{scope}
    \node at (4.1,0) {$+$};
    \begin{scope}[scale=0.75, shift={(7,-0.5)}]
      \ugadgetvert{0}
      \ugadgetvertcolors{9}{0}{9}
      \ugadgetvertcolors{9}{1}{9}
      \ugadgetvertcolors{9}{2}{9}
      \ugadgetvertcolors{9}{3}{9}
      \node at (0.5,-1) {$-1$};
    \end{scope}
  \end{scope}

  \begin{scope}[shift={(0,-\separate)}]
    \node[circle, draw] (d) at (0,0) {};
    \node[circle, draw] (d) at (0,0) {};
    \draw[red, ultra thick] (d) -- ++(-0.5,0.5);
    \draw (d) -- ++(0.5,0.5);
    \draw[red, ultra thick] (d) -- ++(-0.5,-0.5);
    \draw (d) -- ++(0.5,-0.5);

    \node at (1.1,0) {$=$};
    \node at (0,-1) {$-4$};

    \begin{scope}[scale=0.75, shift={(3,-0.5)}]
      \ugadgetvert{0}
      \ugadgetvertcolors{9}{3}{3}
      \ugadgetvertcolors{9}{9}{0}
      \node at (0.5,-1) {$-2$};
    \end{scope}
    \node at (4.1,0) {$+$};
    \begin{scope}[scale=0.75, shift={(7,-0.5)}]
      \ugadgetvert{0}
      \ugadgetvertcolors{9}{0}{0}
      \ugadgetvertcolors{9}{1}{3}
      \ugadgetvertcolors{9}{2}{3}
      \node at (0.5,-1) {$-2$};
    \end{scope}
  \end{scope}

  \begin{scope}[shift={(0,-2*\separate)}]
    \node[circle, draw] (d) at (0,0) {};
    \node[circle, draw] (d) at (0,0) {};
    \draw (d) -- ++(-0.5,0.5);
    \draw[red, ultra thick] (d) -- ++(0.5,0.5);
    \draw[red, ultra thick] (d) -- ++(-0.5,-0.5);
    \draw (d) -- ++(0.5,-0.5);

    \node at (1.1,0) {$=$};
    \node at (0,-1) {$-4$};

    \begin{scope}[scale=0.75, shift={(3,-0.5)}]
      \ugadgetvert{0}
      \ugadgetvertcolors{9}{0}{0}
      \ugadgetvertcolors{9}{1}{2}
      \node at (0.5,-1) {$-2$};
    \end{scope}
    \node at (4.1,0) {$+$};
    \begin{scope}[scale=0.75, shift={(7,-0.5)}]
      \ugadgetvert{0}
      \ugadgetvertcolors{9}{3}{0}
      \ugadgetvertcolors{9}{2}{2}
      \node at (0.5,-1) {$-2$};
    \end{scope}
  \end{scope}

  \begin{scope}[shift={(0,-3*\separate)}]
    \node[circle, draw] (d) at (0,0) {};
    \node[circle, draw] (d) at (0,0) {};
    \draw[red, ultra thick] (d) -- ++(-0.5,0.5);
    \draw[red, ultra thick] (d) -- ++(0.5,0.5);
    \draw[red, ultra thick] (d) -- ++(-0.5,-0.5);
    \draw[red, ultra thick] (d) -- ++(0.5,-0.5);

    \node at (1.1,0) {$=$};
    \node at (0,-1) {$0$};

    \begin{scope}[scale=0.75, shift={(3,-0.5)}]
      \ugadgetvert{0}
      \ugadgetvertcolors{9}{0}{0}
      \ugadgetvertcolors{9}{2}{1}
      \ugadgetvertcolors{9}{9}{2}
      \ugadgetvertcolors{9}{9}{3}
      \node at (0.5,-1) {$s$};
    \end{scope}
    \node at (4.1,0) {$+$};
    \begin{scope}[scale=0.75, shift={(7,-0.5)}]
      \ugadgetvert{0}
      \ugadgetvertcolors{9}{1}{0}
      \ugadgetvertcolors{9}{3}{1}
      \ugadgetvertcolors{9}{1}{2}
      \ugadgetvertcolors{9}{3}{3}
      \node at (0.5,-1) {$-s$};
    \end{scope}
  \end{scope}

\end{tikzpicture}
\caption{The signatures, up to symmetry, of the degree $4$ vertex gadget.}
\label{figure-vertex-gadget-signature}
\end{figure}

Theorem \ref{theorem-degree-4} and the existence of a planar $3$-regular gadget, whose signature is equal to $-4$ times the signature of a single vertex of degree $4$, implies the following.

\begin{theorem}
\label{theorem-degree-3}
  The undirected determinant of cubic planar graphs, with weights in the set $\{-1,-{1\over 2},1\}$, is \#P-complete.
\end{theorem}

\section{The semi-Pfaffian orientation and the computation of the undirected determinant}
\label{section-semi-pfaffian}

The purpose of this section is twofold. Firstly, we prove Theorem \ref{theorem-undirected-det-computable} which states that the undirected determinant is polynomially computable for a reasonable class of cubic planar graphs, which includes the bipartite graphs. This, together with Theorem \ref{theorem-degree-3}, introduces another instance of a tension between $P$ and \#P.
The second aim of this section is the search for tools and ideas which could guide us in our attempts to find polynomially computable analogues of the determinant, hopefully improving the computational strength of the FKT algorithm. In this direction, Definition \ref{definition-semi-pfaffian} introduces the semi-Pfaffian orientation and Definition \ref{definition-tension} introduces the tension of an even cycle in a planar graph. Both of them play the key role in our proof of the polynomial computability of the graphs mentioned above.

\subsection*{\indent Additional preliminaries.}

Let $G=(V,E)$ be a weighted undirected graph.
An {\em orientation} of an undirected graph $G$ is a choice, independently for every edge of $G$, of a direction from one of its end points to the other, the orientation is not considered a part of the graph structure. An undirected graph $G$ on vertices $V=\{1,2,\ldots,n\}$, equipped with an orientation, is represented by a skew symmetric matrix $A=[a_{ij}]_{i,j=1,2,\ldots,n}$ where, for every edge of weight $e$, oriented from $i$ to $j$, we have $a_{ij}=e$ and $a_{ji}=-e$. We put $a_{ij}=0$ if there is no edge $\{i,j\}$ in $G$.

A cycle $c$ in a graph $G$ is {\em even} if it has even length, it is {\em central} if $G\setminus V(c)$ has a perfect matching. An even cycle $c$ is {\em oddly oriented} if for either choice of direction of traversal around $c$, the number of edges of $c$ directed in the direction of the traversal is odd.

Occasionally, it is convenient to abuse the notation and treat a planar graph as if it was a subset of a plane.

Recall that an orientation of the edges of $G$ is {\em Pfaffian} if every even central cycle of $G$ is oddly oriented. At the heart of the FKT algorithm we see two theorems:
\begin{enumerate}
  \item every planar graph admits a Pfaffian orientation
  \item if $A$ is the skew symmetric adjacency matrix, associated with a graph $G$ with Pfaffian orientation, then, up to sign
      $$
        \pfaff A = \perfmatch G
      $$
\end{enumerate}

The standard definition of Pfaffian, see for example Thomas \cite[Section 2]{thomas}, involves a sign function $\sgn_G:pm(G)\arr\{-1,1\}$, which depends on the orientation of $G$ and is defined as
\begin{equation}\label{equation-sign-of-pm}
  \sgn_G(a)=\sgn\left(
  \begin{array}{ccccccc}
    1 & 2 & 3 & 4 & \ldots & 2n-1 & 2n \\
    i_1 & j_1 & i_2 & j_2 & \ldots & i_n & j_n
  \end{array}
  \right)
\end{equation}
where $\sgn$ is the sign of the indicated permutation, and the edges of the perfect matching
$
  a=\{i_1j_1,i_2j_2,\ldots,i_nj_n\}
$
are listed in such a way that every edge $i_kj_k$ is directed from $i_k$ to $j_k$.

The Pfaffian of $A$ is defined as
\begin{equation}\label{equation-pfaffian-def}
  \pfaff A = \sum_{a\in pm(G)}\sgn_G(a)w(a)
\end{equation}
where the sum is taken over all perfect matchings $pm(G)$. Let us recall that the Pfaffian is polynomially computable. In fact the algorithms that compute the determinant tend to translate to algorithms for the Pfaffian.

\subsection*{\indent The semi-Pfaffian orientation.}

\begin{definition}\label{definition-semi-pfaffian}
  An orientation of a graph $G$ is {\em semi-Pfaffian} if every central cycle in $G$ of length $2k$ is oddly oriented if and only if $k$ is odd.
\end{definition}

\begin{remark}\label{remark-semi-pfaffian-examples}
  Not all planar graphs admit semi-Pfaffian orientation, however those which have at most two faces bounded by an odd number of edges, do. This includes the bipartite graphs.
\end{remark}

\subsection*{\indent Cubic planar graphs.}

From now on, we assume that $G$ is a weighted undirected cubic planar graph, with invertible weights of edges. $G$ is equipped with a semi-Pfaffian orientation, and is represented by the skew symmetric matrix $A$.

In such a graph, the complement $\bar a$, of a perfect matching $a\in pm(G)$, in the set of edges $E(G)$, is a cycle cover. Conversely, the complement $\bar c$ of a cycle cover $c\in cc(G)$ is a perfect matching. Note that if $p$ is the product of all the weights of the edges in $G$ then for every $c\in cc(G)$ we have the following.
\begin{equation}\label{equation-wwp}
  w(c)w(\bar c)=p
\end{equation}

Since we are ultimately interested in undirected determinant, we rewrite the formula \ref{equation-pfaffian-def} for Pfaffian into the language of cycle covers, in the case of cubic graphs we have the following.
\begin{equation}\label{equation-pfaffian-computation}
  \pfaff A = \sum_{c\in cc(G)}\sgn_G(\bar c)w(\bar c)
\end{equation}
The Pfaffian is always polynomial time computable -- no assumptions necessary except that $A$ is skew symmetric. On the other hand, Theorem \ref{theorem-degree-3} implies that the undirected determinant,
\begin{equation}\label{equation-determinant-computation}
  \udet G = \sum_{c\in cc(G)}(-1)^{|c|}w(c)
\end{equation}
is \#P-complete even for cubic planar graphs. The $(-1)^n$ factor disappears since a cubic graph has an even number of vertices.

When comparing equations \ref{equation-pfaffian-computation} and \ref{equation-determinant-computation}, we see that the weights $w(c)$ and $w(\bar c)$ are conveniently related by formula \ref{equation-wwp}. In general, the relation between $\sgn_G(\bar c)$ and $(-1)^{|c|}$ is more complicated, however, for some graphs and their orientations, these are nicely related by Proposition \ref{proposition-f-const}.

Every cycle $c$ in a planar graph $G$ yields a decomposition of the plane $P$ into two closed subsets, the bounded one $P_*$ and the unbounded one $P_\infty$. We have $P=P_*\cup P_\infty$ and $c=P_*\cap P_\infty$. Let $G_*=G\cap P_*$ and $G_\infty=G\cap P_\infty$.
Let $v\in c$ be a vertex. We call $v$ an {\em in-vertex} if the unique edge adjacent to $v$ but not in $c$ belongs to $G_*$. Otherwise, we call $v$ an {\em out-vertex}. If $c$ is even then it is bipartite as a subgraph -- let $V_1$, $V_2$ be the bipartition of its vertices.

\begin{definition}\label{definition-tension}
  {\em Tension} of an even cycle in a cubic planar graph is the absolute value of the difference between the number of the out-vertices in $V_1$ and in $V_2$.
\end{definition}

Note that the tension is independent of whether we use the out-vertices or in-vertices in the definition above.

\begin{definition}
  An undirected cubic planar graph $G$ is {\em without tension} if the tension of every even central cycle in $G$ is null.
\end{definition}

\begin{remark}
The graphs mentioned in Remark \ref{remark-semi-pfaffian-examples} -- those with at most two faces which are bounded by an odd number of edges, are without tension.
\end{remark}

\begin{proposition}\label{proposition-f-const}
  If $G$ is an undirected cubic planar graph without tension, equipped with a semi-Pfaffian orientation then the function
  \begin{align*}
    &f_G:cc(G)\arr\{-1,1\} \\
    &f_G(c)=(-1)^{|c|}\sgn_G(\bar c)
  \end{align*}
  is constant.
\end{proposition}
\begin{proof}
It is enough to show that for every $c,d\in cc(G)$ we have
\begin{equation}\label{equation-ff}
  f_G(c)f_G(d)=1
\end{equation}

It is a standard observation that, in a cubic graph, any two cycle covers $c$ and $d$ can be connected by a sequence $c_k\in cc(G)$
$$
  c=c_0,c_1,\ldots,c_r=d
$$
such that at each stage $k=0,1,\ldots,r-1$, the union $\bar c_k\cup\bar c_{k+1}$ contains exactly one cycle. The $\bar c_{k+1}$ is obtained from $\bar c_k$ by replacing those edges in this cycle which belong to $\bar c_k$ with the remaining edges of this cycle -- these remaining edges belong to $\bar d$.

Thus from now on we may assume that $\bar c\cup\bar d$ contains only one cycle, or equivalently, that the symmetric difference $c\triangle d$ is a cycle. By assumption, the unique cycle in $\bar c\cup\bar d$ (or, equivalently, $c\triangle d$) above has null tension.

{\bf The idea of the proof} is to construct recursively a series of modifications of the cycle covers $c$, $d$, and of the ambient graph $G$.
We construct sequences $c_i$, $d_i$ and $G_i$, $i=0,1,\ldots,s$, where $G_0=G$ and $c_i,d_i\in cc(G_i)$, $c_0=c$, $d_0=d$.
The graphs $G_i$ for $i>0$ are going to be multigraphs -- a pair of vertices may be connected by two edges. At each stage we will have
\begin{equation}\label{equation-ff-i}
  f_{G_i}(c_i)f_{G_i}(d_i)=1
\end{equation}
The number of edges in $c_i\triangle d_i$ will be the same as in $c_{i+1}\triangle d_{i+1}$. Both will have no loops and at most one cycle of length exceeding $2$. The longest cycle in $c_{i+1}\triangle d_{i+1}$ will be two edges shorter than the longest one in $c_i\triangle d_i$. The $c_{i+1}\triangle d_{i+1}$ will retain the semi-Pfaffian orientation and null tension. At the end we will have $c_s=d_s$, up to choice of an edge, of a multigraph, between the same vertices.

A single modification is done in three steps which are shown in {\bf Figure \ref{figure-ff-const}}. Below we describe such a modification. To simplify the notation we omit the index $i$ and write $c_1$, $d_1$, $G_1$ for $c_{i+1}$, $d_{i+1}$, $G_{i+1}$.

We have seen above that we may assume that $t=c\triangle d$ is a single cycle of even length. After modifications $c\triangle d$ will have a unique component of length greater than $2$. Since, by assumption, the cycle $t$ has null tension, it must contain two adjacent in-vertices or two adjacent out-vertices. We consider the first case, the other being analogous.

The modification is shown in {\bf Figure \ref{figure-ff-const}}. Below we outline the effect of such modification on the objects we are interested in. The modification:

\begin{itemize}
  \item[(1)] Changes one of the cycle covers, either $c$ or $d$. We obtain $c_1$ and $d_1$ where either $c_1=c$ or $d_1=d$.
  \item[(2)] Shortens the cycle $t$ by two edges. The new cycle $t_1$, a connected component of $c_1\triangle d_1$, retains semi-Pfaffian orientation and null tension.
  \item[(3)] Transforms the ambient graph $G$ into $G_1$ so that the two adjacent in-vertices, denoted $v_2$ and $v_3$ in {\bf Figure \ref{figure-ff-const}}, are connected by two paralel edges in $G_1$.
  \item[(4)] Leaves the product \ref{equation-ff} unchanged, that is $f_G(c)f_G(d)=f_{G_1}(c_1)f_{G_1}(d_1)$.
\end{itemize}

\def\s{1.5.pt}
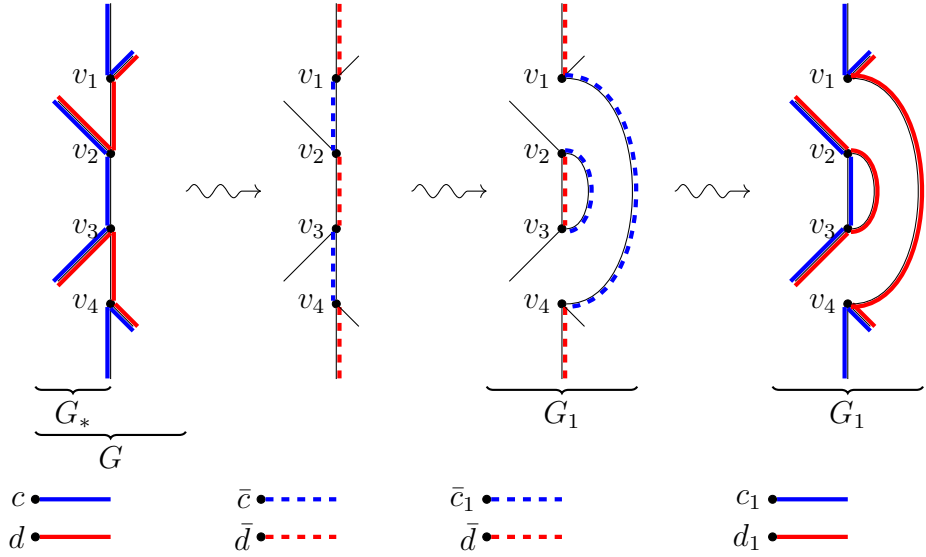
\begin{figure}[H]
\begin{tikzpicture}[scale=1,photon/.style={decorate,decoration={snake,post length=1mm}}]
  \tikzstyle{vertex}=[circle,minimum size=\nodesize,inner sep=0pt,draw=black,fill]
  \tikzstyle{vertexempty}=[circle,minimum size=\nodesize*2,inner sep=0pt,draw=black]

  \begin{scope}[scale=1, shift={(0,0)}]
      \node[vertex] (v1) at (0,4) {};
      \node[vertex] (v2) at (0,3) {};
      \node[vertex] (v3) at (0,2) {};
      \node[vertex] (v4) at (0,1) {};
      \node[vertex] (a8) at (-1,-1.6) {};
      \node[vertex] (a9) at (-1,-2.1) {};
      \draw (a8) node [anchor=east] {$c$};
      \draw[blue, ultra thick] (a8) -- ++ (1,0);
      \draw (a9) node [anchor=east] {$d$};
      \draw[red, ultra thick] (a9) -- ++ (1,0);
      \draw [thick,decoration={brace,mirror,raise=0.1cm},decorate]
        (-1,0) -- (0,0)
        node [pos=0.5,anchor=north,yshift=-0.15cm] {$G_*$};
      \draw [thick,decoration={brace,mirror,raise=0.7cm},decorate]
        (-1,0) -- (1,0)
        node [pos=0.5,anchor=north,yshift=-0.75cm] {$G$};
      \draw (v1) -- (v2) -- (v3) -- (v4) -- ++(0,-1);
      \draw (v1) -- ++(0,1);
      \draw (v2) -- ++(-0.7,0.7);
      \draw (v3) -- ++(-0.7,-0.7);
      \draw[red, ultra thick]
        (v1.south east) -- (v2.north east);
      \draw[red, ultra thick]
        (v2.north) -- ++(-0.7,0.7);
      \draw[red, ultra thick]
        (v3.south) -- ++(-0.7,-0.7);
      \draw[red, ultra thick]
        (v3.south east) -- (v4.north east);

      \draw[blue, ultra thick]
        (v1.north west) -- ++(0,0.95);
      \draw[blue, ultra thick]
        (v2.south west) -- (v3.north west);
      \draw[blue, ultra thick]
        (v4.south west) -- ++(0,-0.95);
      \draw[blue, ultra thick]
        (v2.west) -- ++(-0.7,0.7);
      \draw[blue, ultra thick]
        (v3.west) -- ++(-0.7,-0.7);

      \draw[blue, ultra thick]
        (v1.north) -- ++(0.3,0.3);
      \draw[red, ultra thick]
        (v1.east) -- ++(0.3,0.3);
      \draw
        (v1) -- ++(0.3,0.3);

      \draw[blue, ultra thick]
        (v4.south) -- ++(0.3,-0.3);
      \draw[red, ultra thick]
        (v4.east) -- ++(0.3,-0.3);
      \draw
        (v4) -- ++(0.3,-0.3);

      \draw (v1) node [anchor=east] {$v_1$};
      \draw (v2) node [anchor= east] {$v_2$};
      \draw (v3) node [anchor= east] {$v_3$};
      \draw (v4) node [anchor=east] {$v_4$};
  \end{scope}

  \draw[->,photon] (1,2.5) -- ++(1,0);

  \begin{scope}[scale=1, shift={(3,0)}]
      \node[vertex] (v1) at (0,4) {};
      \node[vertex] (v2) at (0,3) {};
      \node[vertex] (v3) at (0,2) {};
      \node[vertex] (v4) at (0,1) {};
      \node[vertex] (a8) at (-1,-1.6) {};
      \node[vertex] (a9) at (-1,-2.1) {};
      \draw (a8) node [anchor=east] {$\bar c$};
      \draw[dashed, blue, ultra thick] (a8) -- ++ (1,0);
      \draw (a9) node [anchor=east] {$\bar d$};
      \draw[dashed, red, ultra thick] (a9) -- ++ (1,0);

      \draw (v1) -- (v2) -- (v3) -- (v4) -- ++(0,-1);
      \draw (v1) -- ++(0,1);
      \draw (v2) -- ++(-0.7,0.7);
      \draw (v3) -- ++(-0.7,-0.7);
      \draw[dashed, blue, ultra thick]
        (v1.south west) -- (v2.north west);
      \draw[dashed, blue, ultra thick]
        (v3.south west) -- (v4.north west);

      \draw[dashed, red, ultra thick]
        (v1.north east) -- ++(0,0.95);
      \draw[dashed, red, ultra thick]
        (v2.south east) -- (v3.north east);
      \draw[dashed, red, ultra thick]
        (v4.south east) -- ++(0,-0.95);

      \draw
        (v1) -- ++(0.3,0.3);
      \draw
        (v4) -- ++(0.3,-0.3);

%      \draw (v1) to[out=-0,in=0] (v4);
%      \draw (v2) to[out=-0,in=0] (v3);

      \draw (v1) node [anchor=east] {$v_1$};
      \draw (v2) node [anchor= east] {$v_2$};
      \draw (v3) node [anchor= east] {$v_3$};
      \draw (v4) node [anchor=east] {$v_4$};

  \end{scope}

  \draw[->,photon] (4,2.5) -- ++(1,0);

  \begin{scope}[scale=1, shift={(6,0)}]
      \node[vertex] (v1) at (0,4) {};
      \node[vertex] (v2) at (0,3) {};
      \node[vertex] (v3) at (0,2) {};
      \node[vertex] (v4) at (0,1) {};
      \node[vertex] (a8) at (-1,-1.6) {};
      \node[vertex] (a9) at (-1,-2.1) {};
      \draw (a8) node [anchor=east] {$\bar c_1$};
      \draw[dashed, blue, ultra thick] (a8) -- ++ (1,0);
      \draw (a9) node [anchor=east] {$\bar d$};
      \draw[dashed, red, ultra thick] (a9) -- ++ (1,0);
      \draw [thick,decoration={brace,mirror,raise=0.1cm},decorate]
        (-1,0) -- (1,0)
        node [pos=0.5,anchor=north,yshift=-0.15cm] {$G_1$};

      \draw (v2) -- (v3) (v4) -- ++(0,-1);
      \draw (v1) -- ++(0,1);
      \draw (v2) -- ++(-0.7,0.7);
      \draw (v3) -- ++(-0.7,-0.7);

      \draw[dashed, red, ultra thick]
        (v1.north east) -- ++(0,0.95);
      \draw[dashed, red, ultra thick]
        (v2.south east) -- (v3.north east);
      \draw[dashed, red, ultra thick]
        (v4.south east) -- ++(0,-0.95);

      \draw
        (v1) -- ++(0.3,0.3);
      \draw
        (v4) -- ++(0.3,-0.3);

      \draw (v1) to[out=0,in=0] (v4);
      \draw (v2) to[out=0,in=0] (v3);
      \draw[dashed, blue, ultra thick]
        (v1.north east) to[out=0,in=0,distance=1.26cm] (v4.south east);
      \draw[dashed, blue, ultra thick]
        (v2.north east) to[out=0,in=0, distance=0.47cm] (v3.south east);

      \draw (v1) node [anchor=east] {$v_1$};
      \draw (v2) node [anchor= east] {$v_2$};
      \draw (v3) node [anchor= east] {$v_3$};
      \draw (v4) node [anchor=east] {$v_4$};
  \end{scope}

  \draw[->,photon] (7.5,2.5) -- ++(1,0);

  \begin{scope}[scale=1, shift={(9.8,0)}]
      \node[vertex] (v1) at (0,4) {};
      \node[vertex] (v2) at (0,3) {};
      \node[vertex] (v3) at (0,2) {};
      \node[vertex] (v4) at (0,1) {};
      \node[vertex] (a8) at (-1,-1.6) {};
      \node[vertex] (a9) at (-1,-2.1) {};
      \draw (a8) node [anchor=east] {$c_1$};
      \draw[blue, ultra thick] (a8) -- ++ (1,0);
      \draw (a9) node [anchor=east] {$d_1$};
      \draw[red, ultra thick] (a9) -- ++ (1,0);
      \draw [thick,decoration={brace,mirror,raise=0.1cm},decorate]
        (-1,0) -- (1,0)
        node [pos=0.5,anchor=north,yshift=-0.15cm] {$G_1$};

      \draw (v2) -- (v3) (v4) -- ++(0,-1);
      \draw (v1) -- ++(0,1);
      \draw (v2) -- ++(-0.7,0.7);
      \draw (v3) -- ++(-0.7,-0.7);
      \draw[red, ultra thick]
        (v2.north) -- ++(-0.7,0.7);
      \draw[red, ultra thick]
        (v3.south) -- ++(-0.7,-0.7);

      \draw[blue, ultra thick]
        (v1.north west) -- ++(0,0.95);
      \draw[blue, ultra thick]
        (v2.south east) -- (v3.north east);
      \draw[blue, ultra thick]
        (v4.south west) -- ++(0,-0.95);
      \draw[blue, ultra thick]
        (v2.west) -- ++(-0.7,0.7);
      \draw[blue, ultra thick]
        (v3.west) -- ++(-0.7,-0.7);

      \draw[blue, ultra thick]
        (v1.north) -- ++(0.3,0.3);
      \draw[red, ultra thick]
        (v1.east) -- ++(0.3,0.3);
      \draw
        (v1) -- ++(0.3,0.3);

      \draw[blue, ultra thick]
        (v4.south) -- ++(0.3,-0.3);
      \draw[red, ultra thick]
        (v4.east) -- ++(0.3,-0.3);
      \draw
        (v4) -- ++(0.3,-0.3);

      \draw (v1) to[out=0,in=0] (v4);
      \draw (v2) to[out=0,in=0] (v3);
      \draw[red, ultra thick]
        (v1.north east) to[out=0,in=0,distance=1.26cm] (v4.south east);
      \draw[red, ultra thick]
        (v2.north east) to[out=0,in=0, distance=0.47cm] (v3.south east);

      \draw (v1) node [anchor=east] {$v_1$};
      \draw (v2) node [anchor= east] {$v_2$};
      \draw (v3) node [anchor= east] {$v_3$};
      \draw (v4) node [anchor=east] {$v_4$};

  \end{scope}

\end{tikzpicture}
\caption{The modification of $G$ that shortens $t$ and preserves the product $f_G(c)f_G(d)=f_{G_1}(c_1)f_{G_1}(d_1)$}
\label{figure-ff-const}
\end{figure}

In the left drawing we see the adjacent in-vertices $v_2$ and $v_3$ with the surrounding fragments of the cycle covers $c$ and $d$. Possibly swapping the names $c$ and $d$, we may assume that the edge $\{v_2,v_3\}$ belongs to $c$, as it is shown in {\bf Figure \ref{figure-ff-const}}.
The vertices $v_1$ and $v_4$ may be, independently, either the in-vertices or the out-vertices.

In the second drawing we mark the complements $\bar c$ and $\bar d$.

In the third drawing we apply a permutation $\sigma$ to the vertices $\{v_1,v_2,v_3,v_4\}$. The permutation moves the edges $\{v_1,v_2\}$ and $\{v_3,v_4\}$, together with their orientations, and is subject to the following conditions:
\begin{itemize}
  \item[(1)] $\sigma(\{v_1,v_2\})=\{v_2,v_3\}$ so that the old edge $\{v_2,v_3\}$ and the new edge $\sigma(\{v_1,v_2\})$ have the same orientation.
  \item[(2)] $\sigma(\{v_3,v_4\})=\{v_1,v_4\}$ so that the induced orientation of $\{v_1,v_4\}$ agrees with the direction of traversal around $t$ which is induced by the orientations of the even number of the original edges $\{v_1,v_2\}$, $\{v_2,v_3\}$ and $\{v_3,v_4\}$.
\end{itemize}
This modification removes two vertices $v_2$ and $v_3$ from the cycle $t$ which reduces its length by $2$ -- from some $2k$ to $2(k-1)$. The condition (2) implies that the parity of the number of edges directed clockwise (as well as those counterclockwise) changes, therefore the modified cycle $t_1$ retains the semi-Pfaffian orientation.
{\bf Figure \ref{figure-sigma}} lists, up to symmetry, all the possible original orientations of $\{v_1,v_2\}$, $\{v_2,v_3\}$ and $\{v_3,v_4\}$, and the corresponding permutations $\sigma$:

\begin{figure}[H]
\begin{tikzpicture}[scale=1,photon/.style={decorate,decoration={snake,post length=1mm}}]
  \tikzstyle{vertex}=[circle,minimum size=\nodesize,inner sep=0pt,draw=black,fill]
  \tikzstyle{vertexempty}=[circle,minimum size=\nodesize*2,inner sep=0pt,draw=black]

  \begin{scope}[shift={(0,0)}]
    \foreach \x in {1,2,3,4}{
      \node[vertex] (v\x) at (\x,0) {};
      \draw (v\x) node [anchor=north] {$v_\x$};
    }
    \draw[->] (v1) -- (v2);
    \draw[->] (v2) -- (v3);
    \draw[->] (v3) -- (v4);
    \draw[dashed,->] (v4) to[out=90, in=90] (v1);
    \draw[dashed,->] (v2) to[out=90, in=90] (v3);
    \draw (2.5,-1) node {$\sigma = (v_1,v_2,v_3,v_4)$};
  \end{scope}

  \begin{scope}[shift={(4.5,0)}]
    \foreach \x in {1,2,3,4}{
      \node[vertex] (v\x) at (\x,0) {};
      \draw (v\x) node [anchor=north] {$v_\x$};
    }
    \draw[->] (v1) -- (v2);
    \draw[->] (v3) -- (v2);
    \draw[->] (v3) -- (v4);
    \draw[dashed,->] (v1) to[out=90, in=90] (v4);
    \draw[dashed,->] (v3) to[out=90, in=90] (v2);
    \draw (2.5,-1) node {$\sigma = (v_1,v_3)$};
  \end{scope}

  \begin{scope}[shift={(9,0)}]
    \foreach \x in {1,2,3,4}{
      \node[vertex] (v\x) at (\x,0) {};
      \draw (v\x) node [anchor=north] {$v_\x$};
    }
    \draw[->] (v1) -- (v2);
    \draw[->] (v2) -- (v3);
    \draw[->] (v4) -- (v3);
    \draw[dashed,->] (v1) to[out=90, in=90] (v4);
    \draw[dashed,->] (v2) to[out=90, in=90] (v3);
    \draw (2.5,-1) node {$\sigma = (v_1,v_2,v_3,v_4)$};
  \end{scope}

\end{tikzpicture}
\caption{Possible orientations, up to symmetry, and the corresponding permutations $\sigma$.} \label{figure-sigma}
\end{figure}
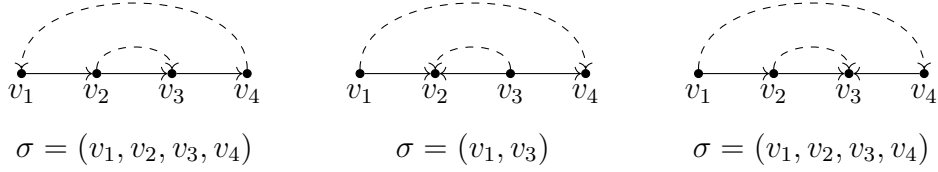

Since in every case shown in {\bf Figure \ref{figure-sigma}}, the permutation $\sigma$ is odd, it changes the sign (see \ref{equation-sign-of-pm}) of the perfect matching $\bar c$, viewed as a summand of the Pfaffian, in \ref{equation-pfaffian-def}. This way we obtain
\begin{equation}\label{equation-const-sgn}
  \sgn_{G_1}(\bar c_1)=-\sgn_G(\bar c)
\end{equation}

In the right drawing of {\bf Figure \ref{figure-ff-const}} we see the cycle covers $c_1$ and $d_1$ -- the complements of $\bar c_1$ and $\bar d$ in the new graph $G_1$. We see that $c_1=c$, however, $d_1$ is different from $d$.

The modification shown in {\bf Figure \ref{figure-ff-const}} changes parity of the number of cycles in the cycle cover $d$ so that
\begin{equation}\label{equation-const-parity}
  (-1)^{|d_1|}=-(-1)^{|d|}
\end{equation}

Combining \ref{equation-const-sgn} and \ref{equation-const-parity} we obtain
\begin{align}\label{equation-ffG1}
\begin{split}
  f_{G_1}(c_1)f_{G_1}(d_1)
  &=(-1)^{|c_1|}\sgn_{G_1}(\bar c_1)(-1)^{|d_1|}\sgn_{G_1}(\bar d_1) \\
  &=(-1)^{|c|}(-\sgn_G(\bar c))(-(-1)^{|d|})\sgn_G(\bar d) \\
  &=f_{G}(c)f_{G}(d)
\end{split}
\end{align}
thus the modification doesn't change the product.

We have deleted two in-vertices from $t$, one from either part of the bipartition of $t$, so that the modified cycle $t_1$ retains the null tension.

We repeat the operation shown in {\bf Figure \ref{figure-ff-const}}, at every step the length of $t$ is reduced by $2$, until at some step $s$ its length is $2$.
At that point we see that every edge in $\bar c_s$ corresponds to an edge in $\bar d_s$, which has the same end points. This correspondence preserves the orientations. Therefore
$\sgn_{G_s}(\bar c_s)=\sgn_{G_s}(\bar d_s)$ and similarly $|c_s|=|d_s|$, so that we obtain
\begin{align}\label{label-ffGr}
\begin{split}
  f_G(c)f_G(d)
  &=\ldots=f_{G_s}(c_s)f_{G_s}(d_s) \\
  &=(-1)^{|c_s|}\sgn_{G_s}(\bar c_s)(-1)^{|d_s|}\sgn_{G_s}(\bar d_s) \\
  &=1
\end{split}
\end{align}
where the dots indicate a sequence of equations given by \ref{equation-ffG1}.
This completes the proof.
\end{proof}

As a corollary, we obtain the following.

\begin{thm}\label{theorem-undirected-det-computable}
  If $G$ is a weighted undirected cubic planar graph with a semi-Pfaffian orientation and without tension, and all the weights in $G$ are nonzero then, up to sign,
  $$
    \udet G = p\cdot\pfaff(A_{inv})
  $$
  where $p$ is the product of the weights of all the edges in $G$ and $A_{inv}$ is the matrix whose nonzero entries are the inverses of the nonzero entries of the skew symmetric adjacency matrix of $G$ with the orientation.
\end{thm}

\begin{proof}
  This is an immediate consequence of Proposition \ref{proposition-f-const} and the identities
  \ref{equation-wwp},
  \ref{equation-pfaffian-computation} and
  \ref{equation-determinant-computation}.
\end{proof}


\begin{thebibliography}{00}

\bibitem{holant-dichotomy} J.-Y. Cai, Z. Fu, H. Guo and T. Williams,
  {\em A Holant dichotomy: is the FKT algorithm universal?}
  2015 IEEE 56th Annual Symposium on Foundations of Computer Science—FOCS 2015, 1259--1276, IEEE Computer Soc., Los Alamitos, CA, 2015.

\bibitem{matchgates-revisited} J.-Y. Cai and A. Gorenstein,
  {\em Matchgates revisited},
  Theory Comput. 10 (2014), 167--197.

\bibitem{CHSS11} S. Chien, P. Harsha, A. Sinclair and S. Srinivasan,
  {\em Almost settling the hardness of noncommutative determinant}, STOC'11—Proceedings of the 43rd ACM Symposium on Theory of Computing, 499--508, ACM, New York, 2011.

\bibitem{salvation} N. de Rugy-Altherre,
  {\em Determinant versus permanent: salvation via generalization?} The nature of computation, 87--96,
  Lecture Notes in Comput. Sci., 7921, Springer, Heidelberg, 2013.

\bibitem{dyson} F.J. Dyson,
  {\em Quaternion determinants}, Helv. Phys. Acta 45 (1972), 289--302.

\bibitem{kasteleyn-a} P.W. Kasteleyn,
  {\em The statistics of dimers on a lattice: I. The number of dimer arrangements on a quadratic lattice},
  Physica, 27 (1961), no. 12, 1209--1225.

\bibitem{kasteleyn-b} P.W. Kasteleyn,
  {\em Graph theory and crystal physics}, 1967 Graph Theory and Theoretical Physics, pp. 43--110, Academic Press, London

\bibitem{mahajan-vinay} M. Mahajan and V. Vinay,
  {\em Determinant: combinatorics, algorithms, and complexity},
  Chicago J. Theoret. Comput. Sci. 1997, Article 5, 26 pp.

\bibitem{fermionant-immanant} S. Mertens and C. Moore,
  {\em The complexity of the fermionant and immanants of constant width}, Theory Comput. 9 (2013), 273--282.

\bibitem{rote} G. Rote,
  {\em Division-free algorithms for the determinant and the Pfaffian: algebraic and combinatorial approaches},
  Computational discrete mathematics, 119--135,
  Lecture Notes in Comput. Sci., 2122, Springer, Berlin, 2001.

\bibitem{temperley-fisher} H.N. Temperley, M.E. Fisher,
  {\em Dimer problem in statistical mechanics—an exact result},
  Philos. Mag. 6, 1961, no. 68, 1061--1063.

\bibitem{thomas} R. Thomas,
  {\em A survey of Pfaffian orientations of graphs}, International Congress of Mathematicians. Vol. III, 963--984, Eur. Math. Soc., Zürich, 2006.

\bibitem{valiant-matchgates} L.G. Valiant,
  {\em Expressiveness of matchgates},
  Theoret. Comput. Sci. 289 (2002), no. 1, 457--471.

\bibitem{valiant-matchgates-quantum} L.G. Valiant,
  {\em Quantum circuits that can be simulated classically in polynomial time},
  SIAM J. Comput. 31 (2002), no. 4, 1229--1254.

\bibitem{valiant-accidental} L.G. Valiant,
   {\em Accidental Algorthims}, 2006 47th Annual IEEE Symposium on Foundations of Computer Science (FOCS'06), Berkeley, CA, USA, 2006, pp. 509--517.

\bibitem{valiant-holographic} L.G. Valiant,
  {\em Holographic algorithms},
  SIAM J. Comput. 37 (2008), no. 5, 1565--1594.

\end{thebibliography}
\end{document}